\documentclass[12pt,letterpaper]{amsart}

\usepackage{ulem}
\usepackage{enumerate}
\usepackage{amssymb,amsthm,amsmath}
\usepackage{mathrsfs}
\usepackage[colorlinks,breaklinks]{hyperref}
\usepackage[margin=1in]{geometry}
\usepackage{braket}
\usepackage{comment}
\usepackage{tikz}
\usetikzlibrary{trees}
\usepackage[edges]{forest}
\usepackage{pgfplots}
\usepackage{graphicx}
\usepackage{caption}
\pgfplotsset{compat=1.11}
\usepgfplotslibrary{fillbetween}
\usetikzlibrary{intersections}

\usetikzlibrary{positioning,arrows}
\usepackage{capt-of}
\usepackage{caption}
\usepackage{float}
\usetikzlibrary{shapes}
\usetikzlibrary{plotmarks}
\tikzset{
  state/.style={circle,draw,minimum size=6ex},
  arrow/.style={-latex, shorten >=1ex, shorten <=1ex}}

\theoremstyle{plain}
\newtheorem{prop}{Proposition}[section]
\newtheorem{lemma}[prop]{Lemma}
\newtheorem{theorem}[prop]{Theorem}
\newtheorem{cor}[prop]{Corollary}

\theoremstyle{definition}
\newtheorem{defi}[prop]{Definition}

\theoremstyle{remark}
\newtheorem{remark}[prop]{Remark}

\newcommand{\Aa}{\textbf{a}}
\newcommand{\Bb}{\textbf{b}}
\newcommand{\Cc}{\textbf{c}}
\newcommand{\Dd}{\textbf{d}}

\newcommand{\Ee}{\textbf{e}}
\newcommand{\Ff}{\textbf{f}}

\newcommand{\C}{\mathbb{C}}
\newcommand{\R}{\mathbb{R}}

\newcommand{\HH}{\mathbb{H}}

\newcommand{\CC}{\mathbb{C}}

\newcommand{\pmat}[4]{\begin{pmatrix} #1&#2\\#3&#4\end{pmatrix}}

\DeclareMathOperator{\dvol}{dvol}
\DeclareMathOperator{\vol}{vol}

\allowdisplaybreaks

\makeatletter
\@namedef{subjclassname@2020}{%
  \textup{2020} Mathematics Subject Classification}
\makeatother

\usepackage{comment}
   
\begin{document}

\title[]{Improved fractal Weyl bounds for convex cocompact hyperbolic surfaces and large resonance-free regions}
\author[L.\@ Soares]{Louis Soares}
\email{louis.soares@gmx.ch}


\begin{abstract}
Let $X$ be a convex cocompact hyperbolic surface, and let $\delta$ denote the Hausdorff dimension of its limit set. Let $N_X(\sigma,T)$ denote the number of resonances of $X$ inside the box $[\sigma, \delta] + i[0,T]$. We prove that for all $\sigma > \delta/2$, we have
\[
N_X(\sigma,T) \ll_\epsilon T^{1 + \delta - 4\left(\sigma - \frac{\delta}{2}\right) + \epsilon}.
\]
This strengthens the previously established ``improved'' fractal Weyl bounds due to Naud~\cite{Naud14} and Dyatlov~\cite{Dya19}. Moreover, this result implies that for every $\epsilon > 0$, there exist resonance-free rectangular boxes of arbitrary height within the strip
\[
\left\{\, s \in \mathbb{C} : \tfrac{3}{4}\delta + \epsilon < \mathrm{Re}(s) < \delta\, \right\}.
\]
Our proof combines Naud's approach~\cite{Naud14} with the refined transfer operator machinery developed by Dyatlov--Zworski~\cite{DyZw18}, as well as a new estimate for oscillatory integrals that arise naturally in our analysis.
\end{abstract}
\maketitle

\section{Introduction and Statement of Results}
In mathematical physics, resonances serve as the primary spectral data in settings where eigenvalues are absent, which is often the case when the underlying geometry is not finite. Over the past decades, substantial research has been devoted to the study of resonances; see~\cite{Zworski2017} for a broad introduction to this subject.

In this paper, we focus on non-elementary, infinite-area, convex cocompact hyperbolic surfaces, that is, surfaces of constant curvature $-1$ that can be decomposed into a compact surface $N$ with geodesic boundary and a finite number of funnel ends glued to $N$. Equivalently, $X$ is a geometrically finite, infinite-area hyperbolic surface without cusps. Fix such a surface $X$ for the remainder of the introduction.

The surface $X$ is isometric to the quotient $\Gamma \backslash \mathbb{H}^2$, where $\mathbb{H}^2$ denotes the hyperbolic plane and $\Gamma \subset \mathrm{PSL}_2(\mathbb{R})$ is a free Fuchsian group. The limit set $\Lambda$ of $X$, i.e., the accumulation set of $\Gamma$-orbits, is a Cantor-like fractal subset of $\partial \HH^2 \cong \R\cup \{ \infty \}$ whose Hausdorff dimension we denote by $\delta\in (0,1).$

Let $\Delta_X$ be the positive Laplacian on $X.$ The resolvent operator
\[
R_X(s) := \left( \Delta_X - s(1 - s) \right)^{-1} \colon L^2(X) \to L^2(X)
\]
is holomorphic for \( \mathrm{Re}(s) > 1 \). By work of Mazzeo--Melrose~\cite{MM87}, it admits a meromorphic continuation to a family of bounded operators
\begin{equation} \label{resolvent_continued}
R_X(s) \colon C_c^\infty(X) \to C^\infty(X),
\end{equation}
defined on all of \( \CC \), with poles of finite rank. For hyperbolic surfaces, a simplified proof was given by Guillop\'{e}-Zworski~\cite{GZ95}. The poles of \( R_X(s) \) are called the \textit{resonances} of \( X \), and their multiplicity is defined as the rank of the corresponding residual operator. We denote by \( \mathcal{R}_X \) the multiset of resonances, counted with multiplicities. All resonances lie in the half-plane \( \mathrm{Re}(s) \le \delta \), and there are no resonances on the line \( \mathrm{Re}(s) = \delta \) except for a simple resonance at \( s = \delta \). We refer the reader to Borthwick's book~\cite{Borthwick_book} for an introduction to the spectral theory of infinite-area hyperbolic surfaces, a subject that has not yet been fully explored.

We are interested in the following resonance counting functions:
\begin{align}
M_X(\sigma,T) &:= \#\left\{ s \in \mathcal{R}_X : \mathrm{Re}(s) \ge \sigma,\, \vert \mathrm{Im}(s) - T \vert \le 1 \right\}, \label{defi:MX} \\
N_X(\sigma,T,H) &:= \#\left\{ s \in \mathcal{R}_X : \mathrm{Re}(s) \ge \sigma,\, \mathrm{Im}(s) \in [T - H, T + H] \right\}, \label{defi:NXTH} \\
N_X(\sigma,T) &:= N_X\left(\sigma, \frac{T}{2}, \frac{T}{2} \right) = \#\left\{ s \in \mathcal{R}_X : \mathrm{Re}(s) \ge \sigma,\, 0 \le \mathrm{Im}(s) \le T \right\}. \label{defi:NX}
\end{align}
Throughout, resonances are counted with multiplicities. From Guillop\'{e}--Lin--Zworski~\cite{GLZ}, we have the upper \textit{fractal Weyl bound} for every fixed $\sigma < \delta$ and $T \ge 1$:
\begin{equation}\label{intro:GLZ}
M_X(\sigma,T) \ll_{\sigma,X} T^\delta.
\end{equation}
By integrating along the imaginary axis, one immediately obtains for all \(1 \le H \le T\),
\[
N_X(\sigma,T,H) \ll_{\sigma,X} H T^\delta \quad \text{and} \quad N_X(\sigma,T) \ll_{\sigma,X} T^{1 + \delta}.
\]
The first fractal Weyl bounds were achieved by Sj\"ostrand in his seminal work on semiclassical Schr\"odinger operators~\cite{Sjoestrand1990}. This work led to the formulation of the \textit{fractal Weyl conjecture} by Lu--Sridhar--Zworski~\cite{Lu_Sridhar_Zworski}, which predicts that the number of resonances near the continuous spectrum grows according to a power law with exponent equal to half the dimension of the classical trapped set. For convex cocompact hyperbolic surfaces, this conjecture would imply that for all sufficiently negative $\sigma \in \mathbb{R}$ and all $T \ge 1$, we have
\begin{equation}\label{FWL}
N_X(\sigma,T) \asymp_{\sigma,X} T^{1+\delta}.
\end{equation}
For numerical evidence in support of this conjecture we refer to~\cite[Chapter~16]{Borthwick_book} and the references therein.

We emphasize that the asymptotic behavior predicted by the fractal Weyl law in \eqref{FWL} is expected to hold only for \textit{sufficiently negative} values of \(\sigma\). For \(\sigma < \delta\) sufficiently close to \(\delta\), this asymptotic no longer holds. In fact, Naud~\cite{Naud14} established the existence of a positive function \(\tau(\sigma)\) in \( (\delta/2, \delta)\) such that
\begin{equation}\label{intro:Naud}
N_X(\sigma,T) \ll_{X} T^{1 + \delta - \tau(\sigma)}.
\end{equation}

Jakobson and Naud~\cite{JN} made the much stronger conjecture asserting that the essential spectral gap of any convex cocompact hyperbolic surface \(X\) is equal to \(\delta/2\), which implies that \(N_X(\sigma,T) = O(1)\) for all \(\sigma > \delta/2\). The heuristic justification for this conjecture is a square-root-type cancellation estimate for the spectral radius of the transfer operator associated with the Schottky representation of \(X\), but a rigorous proof seems to be out of reach given the current state of the art.

Dyatlov~\cite{Dya19} made Naud’s bound explicit by showing that for every $\sigma\in(\delta/2,\delta)$ and all $T \ge 1$,
\begin{equation}\label{intro:dyatlov_m}
M_X(\sigma,T) \ll_{\epsilon,X} T^{2(\delta - \sigma) + \epsilon} = T^{\delta - 2\left(\sigma - \frac{\delta}{2}\right) + \epsilon}.
\end{equation}
One immediately obtains for all \(1 \le H \le T\),
\begin{equation}\label{intro:dyatlov}
N_X(\sigma,T,H) \ll_{\epsilon,X} H T^{\delta - 2\left(\sigma - \frac{\delta}{2}\right) + \epsilon}, \quad 
N_X(\sigma,T) \ll_{\epsilon,X} T^{1 + \delta - 2\left(\sigma - \frac{\delta}{2}\right) + \epsilon}.
\end{equation}
The aim of this paper is to improve the previous two bounds by proving

\begin{theorem}[Main Theorem]\label{main_thm}
Let $X$ be a non-elementary, infinite-area, convex cocompact hyperbolic surface, and let $\delta$ denote the Hausdorff dimension of its limit set. Then for every $\epsilon > 0$ and $0<\eta\le1$, there exists a constant $C = C(\epsilon, \eta, X) > 0$ such that for all $\sigma > \delta/2$, $T \ge 1$ and $T^\eta \le H \le T$, we have
\[
N_X(\sigma, T, H) \le C  H^{1-2\left(\sigma - \frac{\delta}{2}\right)+\epsilon} T^{2\delta - 2\sigma}.
\]
In particular, for every $\epsilon > 0$, there exists a constant $C = C(\epsilon, X) > 0$ such that for all $\sigma > \delta/2$ and $T \ge 1$ we have
\[
N_X(\sigma, T) \le C  T^{1 + \delta - 4\left(\sigma - \frac{\delta}{2}\right) + \epsilon}.
\]
\end{theorem}
A few remarks are in order:

\begin{itemize}
\item While Theorem \ref{main_thm} improves upon \eqref{intro:dyatlov} in the range $T^{\eta} \le H \le T$, our method does not seem to be sufficient to improve the bound for $M_X(\sigma,T)$, due to technical reasons discussed in Section~\ref{sec:outline}.

\item The constants $C$ in the estimates of Theorem~\ref{main_thm} depend intricately on the Schottky data of $X$. We made no attempt to calculate them explicitly.

\item Theorem~\ref{main_thm} holds trivially for elementary surfaces. A hyperbolic surface $X$ is said to be \textit{elementary} if its limit set is finite, or equivalently, if $X$ is a hyperbolic cylinder. In this case, we have $\delta = 0$ and $\mathcal{R}_X$ is explicitly given by the half-lattice $(2\pi i/\ell)\mathbb{Z} - \mathbb{N}_0$, where $\ell > 0$ is the length of the shortest closed geodesic on $X$; see~\cite[Proposition~5.1]{Borthwick_book}.
\end{itemize}

Theorem~\ref{main_thm} implies the existence of resonance-free regions of arbitrarily large height inside the vertical strip
\[
\left\{ s \in \CC : \tfrac{3}{4}\delta + \epsilon < \mathrm{Re}(s) < \delta \right\}.
\]
More precisely, we have

\begin{cor}
Let $X$ be as in Theorem~\ref{main_thm}, and fix $0 < \lambda < \delta$. For every sufficiently small $\epsilon > 0$, there exists a density-one subset $\mathcal{N} \subseteq \mathbb{N}_0$ such that for all $N \in \mathcal{N}$,
\[
\mathcal{R}_X \cap \left( \left[\tfrac{3}{4} \delta + \tfrac{1}{4}\lambda + \epsilon, \delta\right] + i[N, N + N^\lambda] \right) = \varnothing.
\]
Moreover, for any slowly varying, increasing function $f \colon \mathbb{R}^+ \to \mathbb{R}^+$ and any $\epsilon > 0$, there exists a density-one subset $\mathcal{N} \subseteq \mathbb{N}_0$ such that for all $N \in \mathcal{N}$,
\[
\mathcal{R}_X \cap \left( \left[\tfrac{3}{4} \delta + \epsilon, \delta\right] + i[N, N + f(N)] \right) = \varnothing.
\]
\end{cor}

\begin{proof}
Consider the first statement and set
\[
\sigma = \frac{3}{4} \delta  + \frac{1}{4}\lambda + \epsilon.
\]
For $N \in \mathbb{N}$, define
\[
S_{N,\sigma} := \mathcal{R}_X \cap \left( [\sigma, \delta] + i[N, N + N^\lambda] \right),
\]
where resonances are regarded without multiplicity. Let $x>1$, and let $\eta=\eta(x) \in [0,1]$ be the fraction of those $N \in [0, x)$ for which $S_{N,\sigma}$ is non-empty. Since any point can lie in at most $x^\lambda+1\ll x^\lambda$ such sets, we obtain
\[
\left\vert \bigcup_{0 \le N < x} S_{N,\sigma} \right\vert \gg \frac{1}{x^\lambda} \sum_{0 \le N < x} \vert S_{N,\sigma} \vert \gg \eta x^{1 - \lambda}.
\]
On the other hand, Theorem~\ref{main_thm} implies that
\[
\left\vert \bigcup_{0 \le N < x} S_{N,\sigma} \right\vert \le N_X(\sigma, 2x) \ll x^{1 + \delta - 4\left(\sigma - \frac{\delta}{2}\right) + \epsilon}.
\]
Combining both bounds yields
\begin{equation}\label{eq:eta_bound}
\eta \ll x^{\delta - 4\left(\sigma - \frac{\delta}{2}\right) + \lambda + \epsilon}.
\end{equation}
The exponent on the right side of \eqref{eq:eta_bound} is $-3\epsilon < 0$, so $\eta \to 0$ as $x \to \infty$, proving the claim.

For the second statement, slow variation and monotonicity give $f(x)+1=O_\epsilon(x^\epsilon)$ for every $\epsilon>0$. Hence, for large $x$, all the relevant intervals lie below height $2x$, and each point belongs to at most $f(x)+1$ such intervals. Repeating the preceding argument proves the claim.
\end{proof}

\subsection{Outline of the proof}\label{sec:outline} 
We now give an outline of the proof of Theorem \ref{main_thm}. Fix a non-elementary, infinite-area, convex cocompact hyperbolic surface $X$ with limit set $\Lambda$. Resonances for $X$ correspond to the zeros of the Fredholm determinant
$$
f(s) = \det(1 - A(s)),
$$
where $A(s)$ is a holomorphic trace class operator acting on $H^2(\Omega)$, the Hilbert space of holomorphic $L^2$-functions on a neighbourhood $\Omega \subset \CC$ of $\Lambda$. After fixing a Schottky group $\Gamma$ such that $X \cong \Gamma \backslash \HH^2$ (see Section \ref{sec:prelim}), various options for $A(s)$ are available, including:
\begin{itemize}
\item the standard transfer operator $\mathcal{L}_s$ (yielding the Selberg zeta function $f(s)$),
\item its iterates $\mathcal{L}_s^n$ with suitable $n\in \mathbb{N}$, see Naud \cite{Naud14},
\item and the $\tau$-refined operator $\mathcal{L}_{\tau,s}$ introduced by Dyatlov–Zworski \cite{DyZw18}.
\end{itemize}
In this work, we take
$$
A(s) := \mathcal{L}_{\tau_0,\tau_1,s}^2, \qquad \mathcal{L}_{\tau_0,\tau_1,s} := \mathcal{L}_{\tau_1,s}\circ\mathcal{L}_{\tau_0,s},
$$
for suitable resolution parameters $\tau_0, \tau_1 > 0$. These parameters will be chosen depending on the variables $T$ and $H$ in Theorem \ref{main_thm}.

Following Guillop\'{e}--Lin--Zworski~\cite{GLZ}, we introduce an auxiliary parameter $h > 0$ and we let the operators $A(s)$, $s\in \mathbb{C}$, act on the refined functional space $H^2(\Omega(h))$, where $\Omega(h)$ is the complex $h$-neighbourhood of $\Lambda$. The set $\Omega(h)$ is a union of $O(h^{-\delta})$ small Euclidean disks centered on $\mathbb{R}$ with radius $O(h)$. For sufficiently small parameters $h, \tau_0, \tau_1$, the operators
\[
\mathcal{L}_{\tau_0,\tau_1,s} \colon H^2(\Omega(h)) \to H^2(\Omega(h))
\]
are trace class, and the function
\[
f(s) := \det(1 - \mathcal{L}_{\tau_0,\tau_1,s}^2)
\]
can be shown to be a holomorphic multiple of the Selberg zeta function (see Lemma \ref{lem:holomMultiple}).

Thus, the resonance counting function $N_X(\sigma, T, H)$ can be estimated by bounding the number of zeros of $f(s)$ in the rectangular box $[\sigma, \delta] + i[T-H, T+H]$. By using
\[
\log |f(s)| \le \| \mathcal{L}_{\tau_0,\tau_1,s} \|_{\mathrm{HS}, h}^2,
\]
where $\| \cdot \|_{\mathrm{HS}, h}$ denotes the Hilbert--Schmidt norm on $H^2(\Omega(h))$, and applying a variant of Jensen’s formula, this reduces to estimating the integral
\[
\frac{1}{H}\int_{T-H}^{T+H} \| \mathcal{L}_{\tau_0,\tau_1,\sigma + it} \|_{\mathrm{HS}, h}^2 \, dt.
\]
It turns out to be preferable to work with a smoothed version of this integral:
\[
\int_{-\infty}^{+\infty} \varphi_{T,H}(t) \| \mathcal{L}_{\tau_0,\tau_1,\sigma + it} \|_{\mathrm{HS}, h}^2 \, dt,
\]
where $\varphi_{T,H}$ is a non-negative, smooth function supported on $[T-2H, T+2H]$ and equal to $1/H$ on $[T-H, T+H]$ (see Figure \ref{fig:bump-function}).

\begin{figure}[htbp]
  \centering
  \begin{tikzpicture}
    \begin{axis}[
      xlabel={$t$},
      xscale=1.5,
      yscale=0.3,
      xmin=0, xmax=8,
      ymin=0, ymax=0.3,
      axis x line=middle,
      axis y line=middle,
      samples=200,
      domain=0:10,
      xtick={1.4, 2.7, 5.3, 6.6},
      xticklabels={$T-2H$, $T-H$, $T+H$, $T+2H$},
	  ytick={0.2},
      yticklabels={$1/H$},
      tick label style={font=\tiny},
      xlabel style={font=\small},
      ylabel style={font=\small},
      smooth,
      thick,
    ]
      \addplot[black] {1/5 * (0.5 * (tanh(6*(x - 2)) - tanh(6*(x - 6))))};
      \addplot[dotted, thick] coordinates {(2.7,0) (2.7,0.2)};
      \addplot[dotted, thick] coordinates {(5.3,0) (5.3,0.2)};
      \addplot[dotted, thick] coordinates {(0,0.2) (2.7,0.2)};
    \end{axis}
  \end{tikzpicture}
  \caption{The bump function $\varphi_{T,H}$}
  \label{fig:bump-function}
\end{figure}

Since Schottky groups are free on generators indexed by $\mathcal{A} = \{1, \dots, 2m\}$, elements of $\Gamma$ correspond to reduced words $\Aa$. Proposition \ref{prop:HSnorm} gives an explicit expression for $\Vert \mathcal{L}_{\tau_0,\tau_1,s} \Vert_{\mathrm{HS}, h}^2$ as a sum over pairs of concatenated words $(\Aa_0\Aa_1, \Bb_0\Bb_1)$. By the separation lemma, Lemma \ref{lem:separation}, choosing $\tau_0 \approx h$ eliminates the pairs for which $\Aa_0$ and $\Bb_0$ are prefix-incomparable. The remaining pairs are divided according to the suffix-comparability of $\Aa_1$ and $\Bb_1$.

Next, we average each of the terms in the sum over $\vert \mathrm{Im}(s)\vert = t\in [T-H,T+H]$. Specifically, we are led to consider smoothed averaged oscillatory integrals of the form
\begin{equation}\label{intro:osc_aver}
\int_{\Omega(h)} \widehat{\varphi_{T,H}}(\Phi_{\Aa,\Bb}(z)) \, g_{\Aa,\Bb}(z)\, \dvol(z) = \int_{-\infty}^{+\infty} \varphi_{T,H}(t) \left( \int_{\Omega(h)} e^{-it  \Phi_{\Aa,\Bb}(z) } \, g_{\Aa,\Bb}(z)\,\dvol(z)\right)\,  dt,
\end{equation}
where $\Phi_{\Aa,\Bb}(z)$ is a phase function and $g_{\Aa,\Bb}(z)$ is an amplitude depending on $\Aa$ and $\Bb$. When $\Aa \neq \Bb$ and $T^{\eta} \le H \le T$ for some fixed $\eta >0$, Proposition~\ref{prop:FOURIER} gives an inequality that improves upon the trivial bound obtained by the triangle inequality. One key step is the derivative bound $|\Phi_{\Aa,\Bb}'|\asymp \mathcal{D}_{\Aa,\Bb}$ from Proposition~\ref{prop:phaseDerivative}; when $\Aa_1,\Bb_1$ are suffix-incomparable, the derivative is large and in this case the oscillatory integral is small. When $\Aa_1,\Bb_1$ are suffix-comparable, the derivative may be small, but such pairs are few in number and thus can be handled by the triangle inequality.

It is worth noting that, in principle, the same approach could be used to estimate \( M_X(\sigma, T) \). The function \( M_X(\sigma, T) \) is roughly bounded by \( \Vert \mathcal{L}_{\tau_0,\tau_1,s} \Vert_{\mathrm{HS},h}^2 \), which would lead us to consider oscillatory integrals of the form
\begin{equation}\label{intro:osc}
\int_{\Omega(h)} e^{it  \Phi_{\Aa,\Bb}(z) }g(z) \, \mathrm{dvol}(z).
\end{equation}
One might attempt to apply van der Corput's lemma or some variant of it, hoping to capture some form of non-trivial cancellation. However, as $h\searrow 0$, the set \(\Omega_b(h)\) consists of more and more connected components, each having diameter of size at most $O(h)$. These regions seem too small for the oscillatory behaviour to generate substantial cancellation. Fortunately, this obstruction does not arise in the averaged version \eqref{intro:osc_aver}, where cancellation occurs through integration over $t$, rendering the fine geometric structure of $\Omega_b(h)$ largely irrelevant.

Finally, we note that the cancellation obtained in \eqref{intro:osc_aver} should reflect the existence of large resonance-free regions in the strip $\{  \delta/2 < \mathrm{Re}(s)\le \delta\}$. Numerical evidence for such regions can be found in the appendix of \cite{Dya19}.

\subsection{Structure of the Paper}
Section~\ref{sec:prelim} introduces the key ingredients for our proof. Most of the content of this section is well-known or adapted from prior work:
\begin{itemize}
  	\item \ref{sec:hypGeom}: hyperbolic geometry preliminaries,
 	\item \ref{sec:SelbergZeta}: Selberg zeta function,
    \item \ref{sec:SchottkyGroups}: construction of Schottky groups,
    \item \ref{sec:comNotation}: combinatorial notation for indexing words in Schottky groups,
    \item \ref{sec:classicalTransfOper}: definition of the standard transfer operator and the link to the Selberg zeta function,
    \item \ref{sec:partRefOp}: partitions and refined transfer operators,
    \item \ref{sec:boundsForDeriv}: estimates for derivatives of elements of Schottky groups,
    \item \ref{sec:refinedFuncti}: refined functional spaces for the transfer operator,
    \item \ref{sec:boundsForOmega}: estimates concerning the refined set $\Omega(h)$,
    \item \ref{sec:Bergman}: properties of the Bergman kernel.
\end{itemize}

Section~\ref{sec:proofMainThm} contains the actual proof of Theorem~\ref{main_thm}. It is organized into the following subsections:
\begin{itemize}
    \item \ref{sec:separation}: separation lemma,
    \item \ref{sec:phaseder}: bounds for the derivative of the phase function,
    \item \ref{sec:averoscint}: key estimate for averaged oscillatory integrals,
    \item \ref{sec:jensen}: adaptation of Jensen's formula for resonance counting,
    \item \ref{sec:HSnorm}: explicit formula for the Hilbert--Schmidt norm,
    \item \ref{sec:certainSum}: estimate for special sums arising in the final proof steps,
    \item \ref{sec:finish}: putting all the ingredients together.
\end{itemize}

\subsection{Notation}
We use $f(x) \ll g(x)$ or $f(x) = O(g(x))$ to indicate that $|f(x)| \le C|g(x)|$ for some implied constant $C > 0$. We write $f(x) \asymp g(x)$ to mean that $f(x) \ll g(x) \ll f(x)$. We write $f(x) \ll_y g(x)$ to indicate that the implied constant $C$ may depend on a parameter $y$. We write $s = \sigma + it \in \CC$ to mean that $\sigma$ and $t$ are the real and imaginary parts of $s$. Given a finite set $S$, we denote its cardinality by $|S|$ or $\#S$. Finally, the surface $X$ is fixed throughout, and all implied constants may depend on the Schottky data of $X$, even if not explicitly stated.

\subsection{Acknowledgements}
I would like to thank Fr\'{e}d\'{e}ric Naud for his valuable comments, suggestions, and corrections on an earlier version of this paper.

\subsection{Use of AI}
AI language models were used to check arguments and improve exposition. Apart from minor, easy-to-fix issues, GPT-5.6 Terra identified two more delicate issues and provided suggestions for their resolution. First, an earlier version of Lemma~\ref{lem:separation} incorrectly concluded that $\Aa=\Bb$; under the assumptions of the lemma, the correct conclusion is the weaker $\Aa\sim_{\mathrm{pre}}\Bb$ (prefix-comparability). Second, the reduction steps in the proof of Proposition~\ref{prop:main_estimate} require suffix-incomparability, $\Aa\nsim_{\mathrm{suf}}\Bb$, which was previously overlooked. As a result of these corrections, a few isolated changes needed to be made in the proof of Proposition~\ref{prop:main_estimate}. These changes did not affect the main result nor the overall structure of the proof. The author critically reviewed all suggestions and is solely responsible for the mathematical content, including all claims and proofs.

\section{Preliminaries}\label{sec:prelim}

\subsection{Hyperbolic geometry}\label{sec:hypGeom} We start by reviewing some basic facts about hyperbolic geometry, referring the reader to Borthwick's book \cite{Borthwick_book} for a comprehensive discussion. One of the standard models for the hyperbolic plane is the Poincar\'{e} half-plane
$$
\HH^2=\{ x+iy\in \C\ :\ y>0\} 
$$
endowed with its standard metric of constant curvature $-1$,
$$
ds^2=\frac{dx^2+dy^2}{y^2}.
$$ 
The group of orientation-preserving isometries of $(\HH^2, ds)$ is isomorphic to $\mathrm{PSL}_2(\R)$, which acts on the extended complex plane $\overline{\C} = \C \cup \{ \infty\}$ (and hence also on $\HH^2$) by M\"{o}bius transformations
$$
\gamma=\pmat{a}{b}{c}{d}\in \mathrm{PSL}_2(\R),\quad z\in\overline{\C} \Longrightarrow  \gamma(z) = \frac{az+b}{cz+d}.
$$
A non-trivial element \( \gamma \in \mathrm{PSL}_2(\mathbb{R}) \) is classified as:
\begin{itemize}
    \item \textit{hyperbolic}, if \( |\mathrm{tr}(\gamma)| > 2 \); it then has two distinct fixed points on the boundary \( \partial \HH^2 \),
    \item \textit{parabolic}, if \( |\mathrm{tr}(\gamma)| = 2 \); it then has exactly one fixed point on \( \partial \HH^2 \),
    \item \textit{elliptic}, if \( |\mathrm{tr}(\gamma)| < 2 \); it then has a unique fixed point in the interior of \( \HH^2 \).
\end{itemize}

A classical result of Hopf (see for instance \cite[Theorem 2.8]{Borthwick_book}) states that every hyperbolic surface $X$ is a quotient of the hyperbolic plane $\HH^2$ by a \textit{Fuchsian group}, i.e., a discrete subgroup $\Gamma$ of $\mathrm{PSL}_2(\R)$. A Fuchsian group is \textit{torsion-free} if it contains no elliptic elements. It is called \textit{convex cocompact} if it is finitely generated and if it contains neither elliptic nor parabolic elements. Since we only work with torsion-free Fuchsian groups in this paper, it makes no difference whether we work with $\mathrm{PSL}_2(\R)$ or with $\mathrm{SL}_2(\R)$, so we will henceforth stick to $\mathrm{SL}_2(\R)$.

Convex cocompactness has an equivalent geometric characterization, which can be described as follows. The \textit{Nielsen region} of \( X \) is the convex hull \( \widetilde{N} \subset \HH^2 \) of the limit set \( \Lambda \). The \textit{convex core} \( N \) is the quotient \( \Gamma \backslash \widetilde{N} \). It is the smallest closed, non-empty convex subset of \( X \). A surface \( X \) is said to be \textit{convex cocompact} if its convex core is compact. Infinite-area convex cocompact hyperbolic surfaces are isometric to quotients $\Gamma \backslash \HH^2$, where $\Gamma$ is a Schottky group. The construction of Schottky groups is described in Section~\ref{sec:SchottkyGroups} below.

\subsection{Selberg zeta function}\label{sec:SelbergZeta}
Let $\Gamma < \mathrm{PSL}_2(\mathbb{R})$ be a finitely generated Fuchsian group. The set of prime periodic geodesics on $X = \Gamma \backslash \HH^2$ is in bijection with the set $[\Gamma]_{\mathsf{prim}}$ of $\Gamma$-conjugacy classes of primitive hyperbolic elements. For each $[\gamma] \in [\Gamma]_{\mathsf{prim}}$, let $\ell(\gamma)$ denote the corresponding geodesic length.

The \textit{Selberg zeta function} is defined for $\mathrm{Re}(s) > \delta$ as
\begin{equation}\label{selbergZetaDefi}
Z_\Gamma(s) := \prod_{k=0}^\infty \prod_{[\gamma] \in [\Gamma]_{\mathsf{prim}}} \left( 1 - e^{-(s+k)\ell(\gamma)} \right),
\end{equation}
and extends meromorphically to all $s \in \CC$. By Patterson–Perry \cite{Patt_Perry}, the zeros of $Z_\Gamma(s)$ consist of:
\begin{itemize}
  \item topological zeros at $s = -k$ for $k \in \mathbb{N}_0$;
  \item the resonances of $X$, counted with multiplicity.
\end{itemize}
Thus, studying resonances is essentially equivalent to studying the distribution of the zeros of the Selberg zeta function.

We also note that resonances occur in complex conjugate pairs: if $s_0$ is a resonance, then so is $\overline{s_0}$. This follows from the product definition of the Selberg zeta function, which gives $\overline{Z_\Gamma(s)} = Z_\Gamma(\overline{s})$ for $\mathrm{Re}(s) > \delta$ and from uniqueness of analytic continuations, which extends this identity to all $s\in \CC.$

\subsection{Schottky groups}\label{sec:SchottkyGroups}
By a result of Button \cite{Button}, every infinite-area convex cocompact hyperbolic surface $X$ can be realized as a quotient $\Gamma \backslash \HH^2$, where $\Gamma$ is a Schottky group; see also \cite[Theorem~15.3]{Borthwick_book}. We briefly recall the construction:

\begin{itemize}
  \item Define the alphabet $\mathcal{A} = \{1, \dots, 2m\}$ and, for each $a \in \mathcal{A}$, set
  \[
  \overline{a} = 
\begin{cases} 
  a + m &\text{if } a \in \{1, \dots, m\}, \\
  a - m &\text{if } a \in \{m+1, \dots, 2m\}.
\end{cases}
  \]

  \item Choose open disks $D_1, \dots, D_{2m} \subset \CC$ centered on the real axis and in no particular order, with pairwise disjoint closures.

  \item Fix isometries $\gamma_1, \dots, \gamma_{2m} \in \mathrm{SL}_2(\mathbb{R})$ such that for all $a \in \mathcal{A}$,
  \[
  \gamma_a\left( \overline{\CC} \setminus D_{\overline{a}} \right) = \overline{D_a}, \qquad \gamma_{\overline{a}} = \gamma_a^{-1}.
  \]
  (In the notation of \cite[Chapter~15]{Borthwick_book}, this corresponds to $m = r$ and $\gamma_a = S_a^{-1}$.)

  \item Let $\Gamma$ be the subgroup of $\mathrm{SL}_2(\mathbb{R})$ generated by $\gamma_1, \dots, \gamma_{2m}$. Then $\Gamma$ is a free group on $m$ generators; see \cite[Lemma~15.2]{Borthwick_book}.
\end{itemize}

\begin{figure}[H]
\centering
\captionsetup{justification=centering}
\begin{tikzpicture}[xscale=1.5, yscale=1.5]
\draw (-4,0) -- (4,0) node [right, font=\small]  {$  \partial \HH^2 $};
\draw (0,1) node [above, font=\small]  {$ \HH^2 $};
\draw (-3,0) circle [radius = 0.4];
\draw (-1.3,0) circle [radius = 0.8];
\draw (-0.1,0) circle [radius = 0.3];
\draw (1,0) circle [radius = 0.5];
\draw (2.3,0) circle [radius = 0.4];
\draw (3.5,0) circle [radius = 0.2];
\draw (-3,-0.4) node[below,font=\small] {$ D_{1} $};
\draw (-1.3,-0.8) node[below,font=\small] {$ D_{4} $};
\draw (-0.1,-0.3) node[below,font=\small] {$ D_{2} $};
\draw (1,-0.5) node[below,font=\small] {$ D_{3} $};
\draw (2.3,-0.4) node[below,font=\small] {$ D_{5} $};
\draw (3.5,-0.2) node[below,font=\small] {$ D_{6} $};
 \draw [arrow, bend left]  (-3,0.4) to (-1.3,0.8);
 \draw [arrow, bend left]  (-0.1,0.3) to (2.3,0.4);
 \draw [arrow, bend left]  (1,0.5) to (3.5,0.2);
 \draw (-2.2,0.9) node[above,font=\small] {$ \gamma_{1} $};
 \draw (1.1,0.7) node[above,font=\small] {$ \gamma_{2} $};
 \draw (2.4,0.7) node[above,font=\small] {$ \gamma_{3} $};
\end{tikzpicture}
\caption{A configuration of Schottky disks and isometries with $ m=3 $}
\end{figure}

A Fuchsian group $\Gamma$ is called \textit{elementary} if its limit set $\Lambda$ is a finite set. If $\Gamma$ is a Schottky group as above, then it is elementary if and only if $m=1,$ in which case $\Gamma\backslash \HH^2$ is a hyperbolic cylinder.

\textit{Throughout the rest of this paper we assume that $X=\Gamma\backslash \HH^2$ is a convex cocompact quotient where $\Gamma$ is a non-elementary Schottky group with Schottky data $D_1, \dots, D_{2m}$ and $\gamma_1, \dots, \gamma_{2m}$ as above. This assumption will not be repeated in the sequel.}

\subsection{Combinatorial notation for words}\label{sec:comNotation}
We will use a combinatorial notation for elements of the free group $\Gamma$, essentially following Dyatlov–Zworski \cite{DyZw18}.

\begin{itemize}
\setlength\itemsep{1em}

\item Let $\mathcal{A} = \{1, \dots, 2m\}$ be the alphabet. A word $\Aa = a_1 \cdots a_n$ is a string of letters $a_j \in \mathcal{A}$. Let $|\Aa|$ denote the length of the word, i.e., the number of letters in $\Aa.$
We introduce the empty word $\emptyset$ of length zero for technical reasons.

\item A word $\Aa = a_1 \cdots a_n$ is \textit{reduced} if $a_j \ne \overline{a_{j+1}}$ for all $1 \le j \le n-1$. Define $\mathcal{W}_n$ to be the set of reduced words of length $n$:
\[
\mathcal{W}_n = \left\{ a_1 \cdots a_n : a_j \in \mathcal{A},\, a_j \ne \overline{a_{j+1}} \right\}.
\]
Set $\mathcal{W}_0 = \{\emptyset\}$ and $\mathcal{W} = \bigsqcup_{n \ge 0} \mathcal{W}_n$. Also define $\mathcal{W}_{\ge m} = \bigsqcup_{n \ge m} \mathcal{W}_n$ and $\mathcal{W}^\circ = \mathcal{W}_{\ge 1}$. We call $\mathcal{W}^\circ$ the set of non-empty reduced words.

\item For any word $\Aa$, define its reduced form $\mathsf{red}(\Aa) \in \mathcal{W}$ by repeatedly applying the rule $a \overline{a} = \emptyset$. Then $\gamma_\Aa = \gamma_{\mathsf{red}(\Aa)}$ in $\Gamma$.

\item For $\Aa = a_1 \cdots a_n \in \mathcal{W}^\circ$, write $\Aa' = a_1 \cdots a_{n-1}$. The set $\mathcal{W}$ forms a rooted tree with root $\emptyset$, and $\Aa'$ is the parent of $\Aa$.

\item For $\Aa, \Bb \in \mathcal{W}$, write $\Aa \to \Bb$ if either $\Aa = \emptyset$, $\Bb = \emptyset$, or both are non-empty and $a_n \ne \overline{b_1}$; in this case, the concatenation $\Aa\Bb$ is reduced.

  \item For $\Aa = a_1\cdots a_n \in \mathcal{W}^\circ$ and $b \in \mathcal{A}$, write $\Aa \rightsquigarrow b$ if $a_n=b$. We also write $\mathrm{first}(\Aa):=a_1$.

\item Let $z\in \CC.$ We write $\Aa \to z$ if there exists $b \in \mathcal{A}$ such that $\Aa \to b$ and $z \in D_b$.

\item Define the mirror word $\overline{\Aa} = \overline{a_n} \cdots \overline{a_1}$.

\item Write $\Aa \preceq \Bb$ if $\Aa$ is a prefix of $\Bb$, i.e., $\Bb = \Aa\Cc$ for some $\Cc \in \mathcal{W}$. Write $\Aa \prec \Bb$ if $\Aa \preceq \Bb$ and $\Aa \ne \Bb$.

\item We introduce the following notation for prefix- and suffix-comparability:
\[
\Aa\sim_{\mathrm{pre}}\Bb \quad\Longleftrightarrow\quad \Aa\preceq\Bb\ \text{or}\ \Bb\preceq\Aa,
\qquad
\Aa\sim_{\mathrm{suf}}\Bb \quad\Longleftrightarrow\quad \overline{\Aa}\sim_{\mathrm{pre}}\overline{\Bb}.
\]
Thus, $\Aa\sim_{\mathrm{suf}}\Bb$ means that one of $\Aa$ and $\Bb$ is a suffix of the other. We write $\Aa\nsim_{\mathrm{pre}}\Bb$ and $\Aa\nsim_{\mathrm{suf}}\Bb$ for the negations of these relations.

\item The map $\Aa = a_1 \cdots a_n \mapsto \gamma_\Aa = \gamma_{a_1} \cdots \gamma_{a_n}$ gives a bijection between reduced words and elements of $\Gamma$. Moreover:
\[
\gamma_{\Aa\Bb} = \gamma_\Aa \gamma_\Bb, \qquad \gamma_\Aa^{-1} = \gamma_{\overline{\Aa}}, \qquad \gamma_\Aa = \mathrm{id} \iff \Aa = \emptyset.
\]

\item For $\Aa = a_1 \cdots a_n \in \mathcal{W}^\circ$, define the disk $D_\Aa := \gamma_{\Aa'}(D_{a_n})$. Then:
\begin{enumerate}[(i)]
  \item If $\Aa \preceq \Bb$, then $D_\Bb \subset D_\Aa$.
  \item If $\Aa\nsim_{\mathrm{pre}}\Bb$, then $D_\Aa \cap D_\Bb = \emptyset$.
\end{enumerate}
Define the interval
\[
I_\Aa := D_\Aa \cap \mathbb{R}.
\]
Its length $|I_\Aa|$ is equal to the diameter of $D_\Aa$. These intervals contract exponentially: there are constants $C > 0$ and $0 < \theta < 1$ such that
\begin{equation}\label{eq:unifContraction}
|I_\Aa| \le C \theta^{|\Aa|} \quad \text{for all } \Aa \in \mathcal{W}^\circ.
\end{equation}

\item Define the union sets:
\[
D = \bigsqcup_{a \in \mathcal{A}} D_a, \qquad I = \bigsqcup_{a \in \mathcal{A}} I_a.
\]

\item The limit set of $\Gamma$ is given by
\[
\Lambda = \bigcap_{n \ge 1} \bigsqcup_{\Aa \in \mathcal{W}_n} I_\Aa \subset \mathbb{R}.
\]
\end{itemize}

\subsection{The standard transfer operator}\label{sec:classicalTransfOper}
Given a non-empty open set $\Omega \subset \CC$, the \textit{Bergman space} $H^2(\Omega)$ is the Hilbert space of holomorphic functions on $\Omega$ that are square-integrable with respect to Lebesgue measure, i.e.,
\begin{equation}\label{defi_bergman}
H^2(\Omega) := \left\{ f \colon \Omega \to \CC \;\text{holomorphic} \;\middle|\; \|f\| < \infty \right\},
\end{equation}
where the norm is given by
\[
\|f\|^2 := \int_\Omega |f(z)|^2 \, \mathrm{dvol}(z),
\]
and $\mathrm{dvol}$ denotes the standard two-dimensional Lebesgue measure on $\CC$.

Let $D = \bigsqcup_{a \in \mathcal{A}} D_a \subset \CC$ be the union of the Schottky disks introduced in Section~\ref{sec:SchottkyGroups}. We define the holomorphic family of \textit{transfer operators} $\mathcal{L}_s \colon H^2(D) \to H^2(D)$, parametrized by $s \in \CC$, via
\begin{equation}\label{transf_defi}
\mathcal{L}_s f(z) := \sum_{\substack{a \in \mathcal{A} \\ a \to b}} \gamma_a'(z)^s f(\gamma_a(z)), \quad \text{for } z \in D_b.
\end{equation}
Since $\gamma_a'(z) > 0$ for $z \in I_b$, the complex power $\gamma_a'(z)^s$ is holomorphic and unambiguously defined. Explicitly,
\begin{equation}\label{defi:principal_branch}
\gamma_a'(z)^s = \exp\left( s \mathbb{L}(\gamma_a'(z)) \right), \quad \text{where} \quad \mathbb{L}(z) := \log|z| + i\,\mathrm{arg}(z),
\end{equation}
and $\mathrm{arg} \colon \CC \setminus (-\infty, 0] \to (-\pi, \pi)$ is the principal branch of the argument.

The transfer operator $\mathcal{L}_s$ appears, for instance, in \cite[Chapter~15]{Borthwick_book}. We have the following key identity:

\begin{prop}[Fredholm determinant identity]\label{prop:Fredholm_identity}
For all $s \in \CC$, the operator $\mathcal{L}_s$ is trace class on $H^2(D)$, and
\begin{equation}\label{Fredholm_identity_equation}
Z_\Gamma(s) = \det(1 - \mathcal{L}_s).
\end{equation}
\end{prop}

As a consequence, $s \in \CC$ with $s\notin -\mathbb{N}_0$ is a resonance of $X = \Gamma \backslash \HH^2$ if and only if $\mathcal{L}_s$ has $1$ as an eigenvalue with some non-zero eigenfunction $f \in H^2(D)$.

\subsection{Partitions and refined transfer operators}\label{sec:partRefOp}
In this paper, we work with \textit{refined} transfer operators. These are generalizations of the standard transfer operator $\mathcal{L}_s$ introduced by Dyatlov--Zworski \cite{DyZw18}. Given any finite set $Z \subset \mathcal{W}$, define its mirror set as
\[
\overline{Z} := \{ \overline{\Aa} : \Aa \in Z \}.
\]

For any $Z \subset \mathcal{W}$ we define an operator $\mathcal{L}_{Z,s} \colon H^2(D) \to H^2(D)$ by
\begin{equation}\label{generalizedOp}
\mathcal{L}_{Z,s}f(z) := \sum_{\substack{\Aa \in \overline{Z}\\ \Aa \leadsto b}} \gamma_{\Aa'}'(z)^s f(\gamma_{\Aa'}(z)), \quad \text{for } z \in D_b,
\end{equation}
where $\Aa \leadsto b$ means that the last letter of the word $\Aa$ is $b$. This is the same definition as in \cite{DyZw18}, except that we sum over the mirror set $\overline{Z}$ rather than $Z$ itself. Note that for $Z=\mathcal{W}_2$ we have $\mathcal{L}_{Z,s} = \mathcal{L}_{s},$ so these operators generalize the standard transfer operator.

We will use a slightly different notation. For $b\in\mathcal{A}$, put
\[
Y_b(Z):=\{\Aa\in\mathcal{W}: \Aa b\in\overline{Z}\}.
\]
We can then rewrite the refined transfer operator as
\[
\mathcal{L}_{Z,s}f(z) = \sum_{\Aa \in Y_b(Z)} \gamma_{\Aa}'(z)^s f(\gamma_{\Aa}(z)), \quad \text{for } z \in D_b.
\]
This may appear more complicated at first, but will reduce notational complexity in the rest of the paper.

The central concept here is the definition of partitions, as introduced in \cite{DyZw18}.

\begin{defi}[Partitions]
A finite set $Z \subset \mathcal{W}^\circ$ is called a \textit{partition} if there exists $N \in \mathbb{N}$ such that every $\Aa \in \mathcal{W}$ with $|\Aa| \ge N$ has a \textit{unique} prefix $\Bb \in Z$, i.e., $\Bb \preceq \Aa$. Equivalently, $Z$ is a partition if we have the \textit{disjoint} union
\[
\Lambda = \bigsqcup_{\Bb \in Z} (I_{\Bb} \cap \Lambda).
\]
\end{defi}

Examples of partitions are the sets $\mathcal{W}_n$ for $n \ge 2$, in which case we have $\mathcal{L}_{\mathcal{W}_n, s} = \mathcal{L}_s^{n - 1}$. The partitions relevant in this paper are the sets of words discretizing to some resolution $\tau > 0$, which were originally introduced in \cite{DyZw18}:
\begin{equation}\label{defi:Z}
Z(\tau):= \{ \Aa\in \mathcal{W}^{\circ} : \vert I_{\Aa}\vert \le  \tau < \vert I_{\Aa'}\vert \}.
\end{equation}
Here we set $|I_\emptyset|:=+\infty$.
Now, we define the \textit{$\tau$-refined transfer operator} $\mathcal{L}_{\tau,s}$ as
\begin{equation}\label{defiRefinedTO}
\mathcal{L}_{\tau,s}f(z) := \mathcal{L}_{Z(\tau),s}f(z)= \sum_{\Aa\in Y_b(\tau)} \gamma_{\Aa}'(z)^s f(\gamma_{\Aa} (z))\ \mathrm{if}\ z\in D_b,
\end{equation}
where
\begin{equation}\label{defi:Y}
Y_b(\tau) := Y_b(Z(\tau)), \qquad Y(\tau):=\bigcup_{b\in\mathcal{A}}Y_b(\tau)=\overline{Z(\tau)}'.
\end{equation}

The operator $\mathcal{L}_{\tau,s}$ is well-defined if $Y(\tau) \subset \mathcal{W}^\circ$, which holds provided $\tau$ is sufficiently small so that $Z(\tau) \subset \mathcal{W}_{\ge 2}$.

The main reason for using this special family of operators is that we can control the size of the set $Y(\tau)$ as well as the absolute values of the derivatives $\gamma_{\Aa}'$ for $\Aa\in Y(\tau)$ in a uniform way (see Lemma \ref{lem:YZ} below). This is what enables us to obtain explicit exponents in Theorem \ref{main_thm}.

The fundamental fact about partitions is

\begin{lemma}\label{lem:partEigen}
Let $Z \subset \mathcal{W}_{\ge 2}$ be a partition. Then there exists a holomorphic family of trace class operators $T_{Z,s}$ such that
\begin{equation}\label{factWithTraceCLass}
1 - \mathcal{L}_{Z,s} = (1 + T_{Z,s})(1 - \mathcal{L}_s).
\end{equation}
In particular,
\begin{equation}\label{detPartitions}
\det(1 - \mathcal{L}_{Z,s}) = \det(1 + T_{Z,s}) \cdot Z_\Gamma(s).
\end{equation}
\end{lemma}

\begin{remark}
Lemma~\ref{lem:partEigen} extends Lemma~2.4 from Dyatlov--Zworski \cite{DyZw18}. Their result states that any $1$-eigenfunction of $\mathcal{L}_s$ is also a $1$-eigenfunction of $\mathcal{L}_{Z,s}$. Our result is stronger: it shows that the Fredholm determinant $\det(1 - \mathcal{L}_{Z,s})$ is a holomorphic multiple of the Selberg zeta function $Z_\Gamma(s)$. In particular, every zero of $Z_\Gamma(s)$ is also a zero of $\det(1 - \mathcal{L}_{Z,s})$, with at least the same multiplicity. The proof is a straightforward adaptation of the argument in \cite{DyZw18}.
\end{remark}

\begin{proof}
The identity in \eqref{detPartitions} follows from taking determinants on both sides of \eqref{factWithTraceCLass} and then applying Proposition \ref{prop:Fredholm_identity}. For the latter, we argue by induction on $\sum_{\textbf{b}\in Z'} \vert \textbf{b}\vert$. If $Z = \mathcal{W}_2$ then $\mathcal{L}_{Z,s} = \mathcal{L}_s$, so \eqref{factWithTraceCLass} holds true with $T_{\mathcal{W}_2,s} = 0$. If $Z \neq \mathcal{W}_2$, choose a word $\textbf{d}c\in Z$ of maximal length, where $\textbf{d}\in \mathcal{W}^\circ$ and $c\in \mathcal{A}$. 
\begin{figure}[h!]
\centering
\begin{forest}
for tree={
    draw,
    rectangle,
    rounded corners,
    align=center,
    edge={-}
}
[$\emptyset$ 
    [1,
        [11, fill=lightgray]
        [12, 
	        [121, fill=lightgray, scale=0.7]
    	    [122, fill=lightgray, scale=0.7]
        	[123, fill=lightgray, scale=0.7]
        ]
        [14, fill=lightgray]
    ] 
    [2 
        [21, fill=lightgray]
        [22, fill=lightgray]
        [23, fill=lightgray]
    ]
    [3 
        [32, fill=black, text=white
	        [321, fill=lightgray, scale=0.6]
    	    [323, fill=lightgray, scale=0.6]
        	[323, fill=lightgray, scale=0.6]        
        ]
        [33, fill=lightgray]
        [34,
	        [341, fill=lightgray, scale=0.6]
    	    [343, fill=lightgray, scale=0.6]
        	[344, fill=lightgray, scale=0.6]                
        ]
    ] 
    [4 
        [41, fill=lightgray]
        [43, fill=lightgray]
        [44,
        	[441, fill=lightgray, scale=0.7]
    	    [443, fill=lightgray, scale=0.7]
        	[444, fill=lightgray, scale=0.7]        
        ]
    ]
]
\end{forest}
\caption{An example of a partition $Z$, with elements of $Z$ shaded grey in the tree of words. The solid black word is one possible choice of $\textbf{d}$ in the proof of Lemma \ref{lem:partEigen}.}
\label{fig:my-tree}
\end{figure}
Then $Z$ has the form (cf. Figure~\ref{fig:my-tree})
\begin{equation}\label{eq:newform}
Z = (Z_0 \smallsetminus \{\textbf{d}\}) \sqcup \{ \textbf{d}a : a \in \mathcal{A},\, \textbf{d} \to a \},
\end{equation}
where $Z_0$ is a partition containing $\textbf{d}$. By the inductive hypothesis, we have
\begin{equation}\label{indHyp}
1 - \mathcal{L}_{Z_0,s} = (1 + T_{Z_0,s})(1 - \mathcal{L}_s).
\end{equation}
Define the branch operator $R_{\textbf{d},s}\colon H^2(D)\to H^2(D)$ by
\[
R_{\textbf{d},s}(f)(z) :=
\begin{cases}
\gamma_{\overline{\textbf{d}}'}'(z)^s f(\gamma_{\overline{\textbf{d}}'}(z)) & \text{if } z\in D_{\overline{d_1}}, \\
0 & \text{if } z\in D_b,\ b\neq\overline{d_1}.
\end{cases}
\]
For $z\notin D_{\overline{d_1}}$, the expressions for $\mathcal{L}_{Z,s}f(z)$ and $\mathcal{L}_{Z_0,s}f(z)$ coincide. For $z\in D_{\overline{d_1}}$, the replacement of $\textbf{d}$ by its children in \eqref{eq:newform}, the chain rule, and the definition of $\mathcal{L}_s$ give
\[
\mathcal{L}_{Z,s}f(z)=\mathcal{L}_{Z_0,s}f(z)-R_{\textbf{d},s}f(z)+R_{\textbf{d},s}\mathcal{L}_sf(z).
\]
We can rewrite the previous identity as
\[
\mathcal{L}_{Z,s}=\mathcal{L}_{Z_0,s}-R_{\textbf{d},s}(1-\mathcal{L}_s).
\]
Inserting the inductive hypothesis \eqref{indHyp}, we find
\[
1 - \mathcal{L}_{Z,s} = \left(1 + T_{Z_0,s} + R_{\textbf{d},s}\right)(1 - \mathcal{L}_s),
\]
which establishes the relation in \eqref{factWithTraceCLass} with the new operator
\begin{equation}\label{eq:recursion}
T_{Z,s} := T_{Z_0,s} + R_{\textbf{d},s}.
\end{equation}
The operator $R_{\textbf{d},s}$ is a holomorphic trace class weighted composition operator, by the same standard argument as for $\mathcal{L}_s$. Thus $T_{Z,s}$ is again holomorphic and trace class, completing the proof.
\end{proof}

\begin{remark}
Although not important for this paper, it is worth noting that repeated application of \eqref{eq:recursion} yields the identity
\[
T_{Z,s} = \mathcal{L}_{Z^\prec,s},
\]
where $Z^\prec$ denotes the set of all strict prefixes of words in $Z$ of length at least two:
\[
Z^\prec := \left\{ \textbf{b} \in \mathcal{W}_{\ge 2} \mid \exists\, \textbf{a} \in Z \text{ such that } \textbf{b} \prec \textbf{a} \right\}.
\]
In the tree representation of $\mathcal{W}$ (see Figure~\ref{fig:my-tree}), $Z^\prec$ corresponds to all ancestors of $Z$ of length at least two. In the specific case $Z = \mathcal{W}_{n+1}$, we have
\[
Z^\prec = \bigcup_{2 \le m \le n} \mathcal{W}_m,
\]
and Equation \eqref{factWithTraceCLass} becomes the standard identity
\[
1 - \mathcal{L}_s^n = (1 + \mathcal{L}_s + \cdots + \mathcal{L}_s^{n-1})(1 - \mathcal{L}_s).
\]
\end{remark}

It is straightforward to generalize Lemma~\ref{lem:partEigen} to products:

\begin{lemma}\label{lem:holomMultiple}
Let $k \ge 1$ and let $Z_1, \dots, Z_k \subset \mathcal{W}_{\ge 2}$ be a collection of partitions. Then there exists a holomorphic family of trace class operators $T_s$ (depending on $Z_1, \dots, Z_k$) such that
\[
1 - \mathcal{L}_{Z_1,s} \cdots \mathcal{L}_{Z_k,s} = (1 + T_s)(1 - \mathcal{L}_s).
\]
Consequently,
\[
\det\left(1 - \mathcal{L}_{Z_1,s} \cdots \mathcal{L}_{Z_k,s}\right)
\]
is a holomorphic multiple of $Z_\Gamma(s)$.
\end{lemma}

\begin{proof}
We proceed by induction on $k$. For $k = 1$, the statement is Lemma~\ref{lem:partEigen}. Suppose $k > 1$. By the inductive hypothesis, there exists a holomorphic family of trace class operators $\widetilde{T}_s$ such that
\[
1 - \mathcal{L}_{Z_1,s} \cdots \mathcal{L}_{Z_{k-1},s} = (1 + \widetilde{T}_s)(1 - \mathcal{L}_s).
\]
Applying Lemma~\ref{lem:partEigen} to $\mathcal{L}_{Z_k,s}$, we have
\[
1 - \mathcal{L}_{Z_k,s} = (1 + T_{Z_k,s})(1 - \mathcal{L}_s).
\]
Using this, we obtain
\begin{align*}
1 - \mathcal{L}_{Z_1,s} \cdots \mathcal{L}_{Z_k,s}
&= 1 - \left(1 - (1 + \widetilde{T}_s)(1 - \mathcal{L}_s)\right) \mathcal{L}_{Z_k,s} \\
&= 1 - \left(1 - (1 + \widetilde{T}_s)(1 - \mathcal{L}_s)\right) \left(1 - (1 + T_{Z_k,s})(1 - \mathcal{L}_s)\right) \\
&= \left[(1 + \widetilde{T}_s) + (1 + T_{Z_k,s}) - (1 + \widetilde{T}_s)(1 - \mathcal{L}_s)(1 + T_{Z_k,s})\right](1 - \mathcal{L}_s).
\end{align*}
Since sums and products of trace class operators remain trace class, the bracketed expression defines an operator of the form $1 + T_s$ for some trace class $T_s$. Moreover, by a similar argument, $T_s$ depends holomorphically on $s \in \CC$.
\end{proof}

\subsection{Bounds for Schottky groups}\label{sec:boundsForDeriv}
 We now record standard estimates for elements of Schottky groups. We introduce new notation following Magee--Naud \cite{NaudMagee}. For each $a \in \mathcal{A}$, fix a point $o_a \in I_a$. These are treated as fixed reference points. For each $\textbf{a} \in \mathcal{W}^\circ$, fix $b(\textbf{a}) \in \mathcal{A}$ such that $\textbf{a} \to b(\textbf{a})$, and define
\[
o_{\textbf{a}} := o_{b(\textbf{a})},
\]
and let $o_\emptyset$ be any point in $D$. Then for $\textbf{a} \in \mathcal{W}$, define
\begin{equation}\label{defi:Upsilon}
\Upsilon_{\textbf{a}} := \gamma_{\textbf{a}}'(o_{\textbf{a}}).
\end{equation}
From \cite{NaudMagee}, we have the following estimates:

\begin{lemma}[Bounds for derivatives]\label{lem:boundsderivatives}
The following estimates hold with implied constants depending only on $\Gamma$:
\begin{enumerate}[(i)]
\item \label{part:ups} For all $b \in \mathcal{A}$ and all $z \in D_b$ with $\Aa \to b$, we have
\[
|\gamma'_{\Aa}(z)| \asymp \vert I_\Aa\vert \asymp \Upsilon_{\Aa}.
\]

\item For all $\Aa \in \mathcal{W}^\circ$, we have
\[
\Upsilon_{\Aa'} \asymp \Upsilon_{\Aa}.
\]

\item \label{part:mirrorEst} For all $\Aa \in \mathcal{W}^\circ$, we have
\[
\Upsilon_{\overline{\Aa}} \asymp \Upsilon_{\Aa}.
\]

\item \label{part:almostMultip} For all $\Aa, \Bb \in \mathcal{W}^\circ$ with $\Aa \to \Bb$, we have
\[
\Upsilon_{\Aa\Bb} \asymp \Upsilon_{\Aa} \Upsilon_{\Bb}.
\]
\end{enumerate}
\end{lemma}

We also require some estimates concerning the sets $Z(\tau)$ and $Y(\tau)$ defined in \eqref{defi:Z} and \eqref{defi:Y}, respectively.

\begin{lemma}[Estimates for $Z(\tau)$ and $Y(\tau)$]\label{lem:YZ}
There exists $\tau_0>0$ such that for all $\tau\in(0,\tau_0)$, the following estimates hold with constants depending only on $\Gamma$:
\begin{enumerate}[(i)]
\item For all $\Aa \in Z(\tau)$ and all $\Aa \in Y(\tau)$, we have
\[
\Upsilon_{\Aa} \asymp \tau.
\]

\item \label{cardinalityzy} The cardinalities satisfy
\[
|Z(\tau)| \asymp |Y(\tau)| \asymp \tau^{-\delta}.
\]

\item \label{lengthsofwords} For all $\Aa \in Z(\tau)$ and all $\Aa \in Y(\tau)$, we have
\[
|\Aa| \asymp \log(\tau^{-1}).
\]
\end{enumerate}
\end{lemma}

\begin{proof}
The estimates for $Z(\tau)$ appear in \cite{Bourgain_Dyatlov, NaudMagee}. The bounds for $Y(\tau)$ follow from those for $Z(\tau)$ and Lemma~\ref{lem:boundsderivatives}.
\end{proof}

Also helpful is the following lemma:

\begin{lemma}\label{lem:cardsets}
There exists a constant \( C > 0 \), depending only on \( \Gamma \), such that for all \( \tau > 0 \), we have
\[
\#\left\{ \Aa \in \mathcal{W} : \Upsilon_{\Aa} > \tau \right\} \le C(\tau^{-\delta} + 1).
\]
\end{lemma}

\begin{proof}
By \cite[Lemma~10]{CalMag25}, the Frobenius norm of \( \gamma_{\Aa} \) for all \( \Aa \in \mathcal{W} \) satisfies
\[
\Vert \gamma_{\Aa} \Vert \asymp_\Gamma \vert I_\Aa\vert^{-1/2} \asymp_\Gamma \Upsilon_\Aa^{-1/2}.
\]
Moreover, we have the well-known lattice point estimate for all $x > 0$ (see for instance \cite[Theorem~14.22]{Borthwick_book}):
\[
\#\left\{ \gamma \in \Gamma : \Vert \gamma \Vert < x \right\} \ll x^{2\delta} + 1.
\]
Combining these two estimates yields the lemma.
\end{proof}

\subsection{Refined function spaces}\label{sec:refinedFuncti}
In this paper, we adopt the approach of Guillop\'{e}--Lin–Zworski \cite{GLZ}. The idea is to let transfer operators act on \textit{refined} function spaces. To construct these spaces, fix a small parameter $h > 0$ and consider the real $h$-neighbourhood of the limit set:
\begin{equation}\label{defi:realhneighbourhood}
\Lambda(h) := \Lambda + (-h, h).
\end{equation}
The following structural result is known from \cite{GLZ}; see also \cite[Lemma~15.14]{Borthwick_book}.

\begin{lemma}[Structure of $\Lambda(h)$]\label{lem:structureLambdah}
There exists a constant $C > 0$ such that for all sufficiently small $h > 0$, the set $\Lambda(h)$ is a union of at most $C h^{-\delta}$ connected components, each with diameter at most $Ch$.
\end{lemma}
We can therefore write
\begin{equation}\label{LambdaUnion}
\Lambda(h) = \bigsqcup_{l = 1}^{N(h)} I_l(h),
\end{equation}
where $I_l(h) \subset \mathbb{R}$ are pairwise disjoint real intervals, satisfying $|I_l(h)| \le Ch$ and $N(h) \le C h^{-\delta}$. Let $D_l(h) \subset \CC$ be the unique open Euclidean disk centered on the real line such that $D_l(h) \cap \mathbb{R} = I_l(h)$. Define the refined domain:
\begin{equation}\label{defi:omega(h)}
\Omega(h) := \bigsqcup_{l = 1}^{N(h)} D_l(h).
\end{equation}
For each $b \in \mathcal{A}$, let
\begin{equation}\label{defi:omegab}
\Omega_b(h) := \Omega(h) \cap D_b.
\end{equation}
For a partition $Z \subset \mathcal{W}_{\ge N_0+1}$, where $N_0$ is as in Lemma~\ref{lem:contractionInH} below, we define the refined transfer operator $\mathcal{L}_{Z,s} \colon H^2(\Omega(h)) \to H^2(\Omega(h))$ by
\begin{equation}\label{refOpRefSpace}
\mathcal{L}_{Z,s}f(z) := \sum_{\Aa \in Y_b(Z)} \gamma'_{\Aa}(z)^s f(\gamma_{\Aa}(z)), \quad \text{for } z \in \Omega_b(h).
\end{equation}
This is the restriction of \eqref{generalizedOp} to the refined space.

The following lemma is crucial. 

\begin{lemma}[Contraction in $\Omega(h)$]\label{lem:contractionInH}
There exist $N_0 \in \mathbb{N}$ and $h_0 > 0$, depending only on $\Gamma$, such that for all $h \in (0, h_0)$, all $b \in \mathcal{A}$, all $z \in \Omega_b(h)$, and all $\Aa \in \mathcal{W}_{\ge N_0}$ with $\Aa \to b$, we have
\[
\mathrm{dist}(\gamma_{\Aa}(z), \partial \Omega(h)) \geq h/2.
\]
In particular, $\gamma_{\Aa}(\Omega_b(h)) \subset \Omega(h/2)$.
\end{lemma}

\begin{proof}
The argument is the same as in \cite[Lemma~3.2]{Naud14}. After decreasing $h_0$ if necessary, fix $b \in \mathcal{A}$, $z \in \Omega_b(h)$, and $\Aa \in \mathcal{W}^\circ$ with $\Aa \to b$. There exists $p \in \Lambda \cap I_b$ such that $|z-p|\le Ch$. By Lemma~\ref{lem:boundsderivatives} and \eqref{eq:unifContraction}, there are constants $C > 0$ and $0 < \theta < 1$ such that
\[
|\gamma_{\Aa}(z) - \gamma_{\Aa}(p)| \le C \theta^{|\Aa|} h.
\]
Thus, there exists $N_0 \in \mathbb{N}$ independent of $h$ such that for all $|\Aa| \ge N_0$, we have $|\gamma_{\Aa}(z) - \gamma_{\Aa}(p)| \le h/2$. Since $\Lambda$ is $\Gamma$-invariant, we have $\gamma_{\Aa}(p) \in \Lambda$, implying that $\mathrm{dist}(\gamma_{\Aa}(z), \Lambda) \le h/2.$ Thus, $\gamma_{\Aa}(z) \in \Omega(h/2)$, as required.
Since every point of $\Omega(h/2)$ has distance at least $h/2$ from $\partial\Omega(h)$, this proves the stated estimate.
\end{proof}

A direct consequence of this is

\begin{lemma}[Well-defined action on refined space]\label{lem:welldefinedness}
Let $N_0 \in \mathbb{N}$ and $h_0 > 0$ be as in Lemma~\ref{lem:contractionInH}. For every finite set $Z \subset \mathcal{W}_{\ge N_0+1}$ and every $h \in (0, h_0)$, the operator
\begin{equation}\label{defi:refTOonRefS}
\mathcal{L}_{Z,s} \colon H^2(\Omega(h)) \to H^2(\Omega(h)), \quad s \in \CC
\end{equation}
defined by Equation~\eqref{refOpRefSpace} is well-defined and of trace class.

Moreover, there exists $\tau_0 > 0$ such that for all $\tau \in (0, \tau_0)$ and all $h \in (0, h_0)$, the estimate
\[
\mathrm{dist}(\gamma_{\Aa}(z), \partial \Omega(h)) \ge h/2
\]
holds for all $b \in \mathcal{A}$ and all $\Aa \in Y_b(\tau).$ Consequently, the $\tau$-refined operator
\begin{equation}\label{defi:refTOonRefS_2}
\mathcal{L}_{\tau,s} \colon H^2(\Omega(h)) \to H^2(\Omega(h)), \quad s \in \CC
\end{equation}
is also well-defined and trace class.
\end{lemma}

\begin{proof}
The first part follows from Lemma~\ref{lem:contractionInH}. For the second, use Part~\eqref{lengthsofwords} of Lemma~\ref{lem:YZ} to choose $\tau_0 > 0$ such that $Y(\tau) \subset \mathcal{W}_{\ge N_0}$ for all $\tau \in (0, \tau_0)$. 
\end{proof}

\subsection{Some bounds for $\Omega(h)$}\label{sec:boundsForOmega}
We now record several key estimates concerning the refined set $\Omega(h)$.

\begin{lemma}[Volume bound]\label{lem:vol}
There exist constants $C > 0$ and $h_0>0$, depending only on $\Gamma$, such that for all $h \in (0,h_0)$, 
\[
\vol(\Omega(h)) \le C h^{-\delta + 2}.
\]
\end{lemma}

\begin{proof}
By Lemma~\ref{lem:structureLambdah}, $\Omega(h)$ is a disjoint union of at most $C h^{-\delta}$ disks, each of radius $\le C h$. Since each disk has area $\le \pi C^2 h^2$, this yields the stated bound.
\end{proof}

\begin{lemma}[Taylor expansion of $\mathbb{L}(\gamma_{\Aa}'(z))$ in $\Omega(h)$]\label{lem:taylor}
Let $\Aa \in \mathcal{W}^\circ$, $b \in \mathcal{A}$, and $z = x + i y \in \Omega(h)$ with $\Aa \to b$ and $z \in \Omega_b(h)$. Then, for all sufficiently small $h > 0$ in terms of $\Gamma$,
\[
\mathbb{L}(\gamma_{\Aa}'(z)) = \log(\gamma_{\Aa}'(x))  + O(h),
\]
where the implied constant depends only on $\Gamma$.
\end{lemma}

\begin{proof}
Write $\gamma_{\Aa}$ as a Möbius transformation:
\[
\gamma_{\Aa}(x) = \frac{a_{\Aa} x + b_{\Aa}}{c_{\Aa} x + d_{\Aa}}, \quad \gamma_{\Aa}'(x) = \frac{1}{(c_{\Aa} x + d_{\Aa})^2},
\]
where
\[
\gamma_{\Aa} = \pmat{a_{\Aa}}{b_{\Aa}}{c_{\Aa}}{d_{\Aa}}\in \mathrm{SL}_2(\R).
\]
Its $n$-th derivative satisfies
\[
\gamma_{\Aa}^{(n)}(x) = n! \left( \frac{1}{x_{\Aa} - x} \right)^{n - 1} \gamma_{\Aa}'(x), \quad \text{where } x_{\Aa} := \gamma_{\Aa}^{-1}(\infty) = -\frac{d_{\Aa}}{c_{\Aa}}.
\]
Since $z \in D_b$ and $x_{\Aa} \in D_{\overline{a}_n}$ (with $\overline{a}_n \neq b$), we have $|z - x_{\Aa}| \asymp_\Gamma 1 $. In particular, for all $x \in I_b$ and all $n \in \mathbb{N}$,
\begin{equation}\label{eq:bound_gamma_derivative}
|\gamma_{\Aa}^{(n)}(x)| \le n! C^n \gamma_{\Aa}'(x)
\end{equation}
for some $C > 0$ depending only on $\Gamma$.

Now let $z = x + i y \in \Omega_b(h)$, so $|y| = O(h)$. Taylor-expand $\gamma_{\Aa}'(z)$ around $iy$:
$$
\gamma_{\Aa}'(z)  = \gamma_{\Aa}'(x+iy)= \sum_{n=0}^\infty  \frac{\gamma_{\Aa}^{(n+1)}(x)}{n!} (iy)^n.
$$
Using \eqref{eq:bound_gamma_derivative}, we see that for all $h>0$ sufficiently small in terms of $\Gamma$,
\begin{equation}\label{afterTayloer}
\gamma_{\Aa}'(z) = \gamma_{\Aa}'(x) +  O( y  \gamma_{\Aa}'(x) ) = \gamma_{\Aa}'(x) \left(  1  +  O( y ) \right).
\end{equation}
Taking logarithms using the principal branch $\mathbb{L}$ from \eqref{defi:principal_branch}, we get
\[
\mathbb{L}(\gamma_{\Aa}'(z)) = \log(\gamma_{\Aa}'(x)) + \mathbb{L}\left(1 + O(y)\right),
\]
where we used the rule
$$
\mathbb{L}(az) = \log(a) + \mathbb{L}(z)
$$
valid for all $a\in \R_{>0}$ and $z\in \C\setminus (-\infty,0].$ Using the Taylor series $\mathbb{L}(1 + u) = u - \frac{u^2}{2} + \dots$ for $\vert u\vert < 1$, we conclude, after possibly shrinking $h>0$, that
\[
\mathbb{L}(\gamma_{\Aa}'(z)) = \log(\gamma_{\Aa}'(x))  + O(h),
\]
as claimed.
\end{proof}

Using the previous result, we obtain

\begin{lemma}[Bound for complex powers of derivatives]\label{lem:crucial_estimate}
There exist $C > 0$ and $h_0>0$, depending only on $\Gamma$, such that for all $h\in(0,h_0)$, all $b \in \mathcal{A}$, all $\Aa \in \mathcal{W}^\circ$ with $\Aa \to b$, and all $s = \sigma + i t \in \CC$ with $\sigma > 0$,
\[
z \in \Omega_b(h) \Longrightarrow |\gamma_{\Aa}'(z)^s| \le (C \Upsilon_{\Aa})^{\sigma} \exp(C h |t|).
\]
\end{lemma}

\begin{proof}
Let $z = x + i y \in \Omega_b(h)$. By Lemma~\ref{lem:taylor}, we have
\begin{align*}
\gamma_{\Aa}'(z)^s &= \gamma_{\Aa}'(z)^{\sigma} \exp\left(i t \mathbb{L}(\gamma_{\Aa}'(z))\right) \\
&= \gamma_{\Aa}'(z)^{\sigma} \exp\left(i t \left(\log(\gamma_{\Aa}'(x)) +  O(h)\right)\right).
\end{align*}
Taking absolute values yields
\[
|\gamma_{\Aa}'(z)^s| = |\gamma_{\Aa}'(z)|^{\sigma} \exp\left(O(h |t|)\right).
\]
Using $|y| = O(h)$, as well as $\gamma_{\Aa}'(z) \asymp \Upsilon_{\Aa}$ from Lemma~\ref{lem:boundsderivatives}, we conclude that
\[
|\gamma_{\Aa}'(z)^s| \le (C \Upsilon_{\Aa})^{\sigma} \exp(O(h |t|)),
\]
as claimed.
\end{proof}

\begin{remark}
\begin{itemize}
\item Lemma~\ref{lem:crucial_estimate} motivates the choice $h := |t|^{-1}$. This ensures that exponential terms of the form $\exp(C h |t|)$ remain uniformly bounded, allowing us to eliminate exponential growth in estimates. This is the key benefit of working in the refined function spaces $H^2(\Omega(h))$.

\item Lemma~\ref{lem:contractionInH} implies that for any finite collection of partitions $Z_1, Z_2, \dots, Z_k \subset \mathcal{W}_{\ge N_0+1}$, and for sufficiently small $h > 0$, the operator
\[
\mathcal{L}_{Z_1,s} \cdots \mathcal{L}_{Z_k,s} \colon H^2(\Omega(h)) \to H^2(\Omega(h))
\]
is well-defined and trace class. By carefully applying the Lefschetz fixed point formula (cf.~\cite[Lemma~15.9]{Borthwick_book}), one finds that the traces of its powers, and hence the Fredholm determinant $\det(1 - \mathcal{L}_{Z_1,s} \cdots \mathcal{L}_{Z_k,s})$, are independent of $h$. By Lemma~\ref{lem:holomMultiple}, this determinant equals $Z_\Gamma(s)$ times an entire function.
\end{itemize}
\end{remark}

\subsection{Properties of Bergman kernels}\label{sec:Bergman}
We now state some basic facts about Bergman kernels, which play a crucial role in our proof of Theorem~\ref{main_thm}. For an in-depth account of the material given here, see \cite{Krantz}.

Let $\Omega \subset \CC$ be a non-empty bounded (possibly disconnected) open set and let $H^2(\Omega)$ denote the associated Bergman space. As a closed subspace of $L^2(\Omega)$, $H^2(\Omega)$ is separable and admits an orthonormal basis $(\varphi_n)$. Thus, any $f \in H^2(\Omega)$ has the expansion
\[
f(z) = \sum_n c_n(f) \varphi_n(z), \quad
c_n(f) = \int_\Omega f(w) \overline{\varphi_n(w)} \dvol(w),
\]
which converges absolutely and uniformly on compact subsets of $\Omega$.

By the Riesz representation theorem, there exists a unique function $\overline{B_\Omega(z,\cdot)} \in H^2(\Omega)$, called the \textit{Bergman reproducing kernel}, such that
\begin{equation}\label{berg:int}
f(z) = \int_\Omega B_\Omega(z,w) f(w) \dvol(w).
\end{equation}
We now derive an explicit expression for $B(z,w)$. For any fixed $z\in \Omega$, we can expand $\overline{B(z,w)}$ as follows:
\begin{equation}\label{berg:exp}
\overline{B(z,w)} = \sum_n c_n \varphi_n(w), 
\end{equation}
where 
\begin{equation}\label{berg:coeff}
c_n = \int_\Omega \overline{B(z,w)} \, \overline{\varphi_n(w)} \dvol(w).
\end{equation}
Inserting \eqref{berg:coeff} and \eqref{berg:exp} into \eqref{berg:int}, and using the uniqueness of the Bergman kernel, we obtain
\begin{equation}\label{eq:formulaBergmann}
B(z,w) = \sum_n \varphi_n(z) \overline{\varphi_n(w)},
\end{equation}
where the series is uniformly convergent on compact subsets of $ \Omega\times \Omega. $

\begin{lemma}[Basic properties of the Bergman kernel]\label{lem:BergmanBasicProp}
The following hold:
\begin{enumerate}[(i)]
\item \label{part:bergmanDiffComp} If $z, w$ lie in distinct connected components of $\Omega$, then $B_\Omega(z,w) = 0$.

\item \label{part:bergmanIneq} For all $z, w \in \Omega$,
\[
|B_\Omega(z,w)|^2 \le B_\Omega(z,z) B_\Omega(w,w).
\]

\item \label{part:bergmanVariational} Variational formula:
\[
B_\Omega(z,z) = \sup\left\{ |f(z)|^2 : f \in H^2(\Omega),\, \|f\| = 1 \right\}.
\]

\item \label{part:bergmanComp} Comparison property: If $z \in \Omega_1 \subset \Omega_2$, then
\[
B_{\Omega_2}(z,z) \le B_{\Omega_1}(z,z).
\]

\item \label{part:bergmanDisk} Explicit formula for disks:
For $D(z_0,r) := \{ z \in \CC : |z - z_0| < r \}$,
\[
B_{D(z_0,r)}(z,w) = \frac{r^2}{\pi \left( r^2 - (z - z_0)(\overline{w} - \overline{z}_0) \right)^2}.
\]
\end{enumerate}
\end{lemma}

\begin{proof}
Parts \eqref{part:bergmanDiffComp}--\eqref{part:bergmanVariational} follow from \eqref{eq:formulaBergmann} and the Cauchy--Schwarz inequality. For \eqref{part:bergmanComp}, restrict a function in $H^2(\Omega_2)$ to $\Omega_1$ and apply \eqref{part:bergmanVariational}. To prove \eqref{part:bergmanDisk}, take the orthonormal basis
\[
\varphi_n(z) = \sqrt{\frac{n+1}{\pi r^2}} \left( \frac{z - z_0}{r} \right)^n, \quad n \in \mathbb{N}_0
\]
and use the formula in \eqref{eq:formulaBergmann}.
\end{proof}

\begin{lemma}[Upper bound for the Bergman kernel]\label{lem:UpperBoundBergman}
Define
\[
\mathrm{dist}(z, \partial\Omega) := \inf_{z' \in \partial\Omega} |z - z'|.
\]
Then for all $z, w \in \Omega$,
\begin{equation}\label{UpperBoundBergmanineq}
|B_\Omega(z,w)| \le \frac{1}{\pi \, \mathrm{dist}(z, \partial\Omega) \, \mathrm{dist}(w, \partial\Omega)}.
\end{equation}
\end{lemma}

\begin{proof}
Choose radii $r_z = \mathrm{dist}(z, \partial\Omega)$ and $r_w = \mathrm{dist}(w, \partial\Omega)$ so that $D(z, r_z), D(w, r_w) \subset \Omega$. By parts \eqref{part:bergmanIneq}, \eqref{part:bergmanComp}, and \eqref{part:bergmanDisk} of Lemma~\ref{lem:BergmanBasicProp}, we estimate:
\[
|B_\Omega(z,w)|^2 \le B_\Omega(z,z) B_\Omega(w,w)
\le B_{D(z, r_z)}(z,z) B_{D(w, r_w)}(w,w)
\le \frac{1}{\pi r_z^2} \cdot \frac{1}{\pi r_w^2},
\]
which proves the claim.
\end{proof}

\section{Proof of Main Theorem}\label{sec:proofMainThm}
This section is devoted to the proof of Theorem~\ref{main_thm}, and is divided into several parts. Recall that $X = \Gamma \backslash \HH^2$ is a convex cocompact hyperbolic surface, and that we fix a Schottky representation for $\Gamma$ as in Section~\ref{sec:SchottkyGroups}. Throughout, we use the notations introduced in Section~\ref{sec:prelim}.

\subsection{Separation lemma}\label{sec:separation}
We begin with a separation lemma adapted from \cite[Lemma~4.4]{JN}, reflecting the total discontinuity of the limit set $\Lambda$. Recall that $\Aa\sim_{\mathrm{pre}}\Bb$ means that one of $\Aa$ and $\Bb$ is a prefix of the other, including equality, and $\Aa\sim_{\mathrm{suf}}\Bb$ means that $\overline{\Aa}\sim_{\mathrm{pre}}\overline{\Bb}$.

\begin{lemma}[Separation Lemma]\label{lem:separation}
There exist constants $\widetilde{C} > 0$ and $h_0 > 0$, depending only on $\Gamma$, such that for all $0 < h < h_0$ and $\tau \ge \widetilde{C} h$, the following holds. For all $b_1,b_2 \in \mathcal{A}$, all words $\Aa \in Y_{b_1}(\tau)$ and $\Bb \in Y_{b_2}(\tau)$, and all points $z_1 \in D_{b_1}$ and $z_2 \in D_{b_2}$, we have:
\[
\gamma_{\Aa}(z_1),\, \gamma_{\Bb}(z_2) \text{ lie in the same connected component of } \Omega(h) \quad \Longrightarrow \quad \Aa\sim_{\mathrm{pre}}\Bb.
\]
\end{lemma}

\begin{proof}
Fix $b_1,b_2 \in \mathcal{A}$, $\Aa \in Y_{b_1}(\tau)$, $\Bb \in Y_{b_2}(\tau)$, and $z_1 \in D_{b_1}$, $z_2 \in D_{b_2}$. Suppose that neither $\Aa \preceq \Bb$ nor $\Bb \preceq \Aa$, yet $\gamma_{\Aa}(z_1)$ and $\gamma_{\Bb}(z_2)$ lie in the same component of $\Omega(h)$. We aim for a contradiction when $\tau \ge \widetilde{C} h$ for some sufficiently large constant $\widetilde{C}$.

By Lemma~\ref{lem:structureLambdah}, each component of $\Omega(h)$ has diameter $\le Ch$, so
\begin{equation}\label{diamH}
|\gamma_{\Aa}(z_1) - \gamma_{\Bb}(z_2)| \le Ch.
\end{equation}

Let $\Cc$ be the maximal common prefix of $\Aa$ and $\Bb$. If $\Cc=\emptyset$, then $\gamma_{\Aa}(z_1)$ and $\gamma_{\Bb}(z_2)$ lie in distinct Schottky disks, hence their distance is at least the positive minimum $K$ of the distances between distinct Schottky disks. This contradicts \eqref{diamH} for sufficiently small $h$. We may therefore assume that $\Cc\in\mathcal{W}^\circ$. Since $\Aa$ and $\Bb$ are prefix-incomparable, we can write $\Aa=\Cc\Aa_1$ and $\Bb=\Cc\Bb_1$, where $\Aa_1,\Bb_1\in\mathcal{W}^\circ$ have distinct first letters. Thus
\begin{equation}\label{ineq:largerK2}
|\gamma_{\Aa_1}(z_1) - \gamma_{\Bb_1}(z_2)| \ge K.
\end{equation}
Now, using the identity for any Möbius transformation $\gamma$,
\[
|\gamma(u) - \gamma(v)| = |\gamma'(u)|^{1/2} |\gamma'(v)|^{1/2} |u - v|,
\]
applied to $\gamma = \gamma_\Cc$ and $u = \gamma_{\Aa_1}(z_1)$, $v = \gamma_{\Bb_1}(z_2)$, we obtain together with Lemmas~\ref{lem:boundsderivatives} and \ref{lem:YZ}:
\begin{align*}
|\gamma_{\Aa}(z_1) - \gamma_{\Bb}(z_2)|
&= |\gamma_\Cc'(\gamma_{\Aa_1}(z_1))|^{1/2} |\gamma_\Cc'(\gamma_{\Bb_1}(z_2))|^{1/2} |\gamma_{\Aa_1}(z_1) - \gamma_{\Bb_1}(z_2)| \\
&\gg \Upsilon_\Cc \cdot |\gamma_{\Aa_1}(z_1) - \gamma_{\Bb_1}(z_2)| \\
&\gg \Upsilon_\Aa \cdot |\gamma_{\Aa_1}(z_1) - \gamma_{\Bb_1}(z_2)| \\
&\gg \tau \cdot |\gamma_{\Aa_1}(z_1) - \gamma_{\Bb_1}(z_2)|.
\end{align*}
For the second estimate we used that $\Upsilon_\Cc \gg \Upsilon_\Aa$ for any prefix $\Cc \prec \Aa$. From \eqref{ineq:largerK2}, we now conclude that
\[
|\gamma_{\Aa}(z_1) - \gamma_{\Bb}(z_2)| \ge C_0 \tau,
\]
for some constant $C_0 > 0$. Combining this with \eqref{diamH} gives $\tau \le \frac{C}{C_0} h$, which contradicts the assumption $\tau \ge \widetilde{C} h$ if $\widetilde{C} > \frac{C}{C_0}$. This completes the proof.
\end{proof}

The next lemma will be used to control pairs of prefix- and suffix-comparable words in $Y(\tau)$.

\begin{lemma}[Bounded transitions]\label{lem:boundedTransitions}
There exists a finite set $F\subset\mathcal{W}^\circ$, depending only on $\Gamma$, such that for every sufficiently small $\tau>0$ and all distinct words $\Aa,\Bb\in Y(\tau)$:
\begin{itemize}
\item if $\Aa\sim_{\mathrm{pre}}\Bb$, then either $\Bb=\Aa\textbf{v}$ or $\Aa=\Bb\textbf{v}$ for some $\textbf{v}\in F$;
\item if $\Aa\sim_{\mathrm{suf}}\Bb$, then either $\Bb=\textbf{v}\Aa$ or $\Aa=\textbf{v}\Bb$ for some $\textbf{v}\in F$.
\end{itemize}
\end{lemma}

\begin{proof}
In each case, after interchanging $\Aa$ and $\Bb$ if necessary, we have either $\Bb=\Aa\textbf{v}$ or $\Bb=\textbf{v}\Aa$ with $\textbf{v}\in\mathcal{W}^\circ$, where the concatenation is reduced. By Lemmas~\ref{lem:boundsderivatives} and \ref{lem:YZ},
\[
\tau\asymp\Upsilon_{\Bb}\asymp\Upsilon_{\Aa}\Upsilon_{\textbf{v}}\asymp\tau\Upsilon_{\textbf{v}}.
\]
Hence $\Upsilon_{\textbf{v}}\asymp1$. Since $\Upsilon_{\textbf{v}}\asymp|I_{\textbf{v}}|\ll\theta^{|\textbf{v}|}$ by Lemma~\ref{lem:boundsderivatives} and \eqref{eq:unifContraction}, the length of $\textbf{v}$ is bounded in terms of $\Gamma$. Taking $F$ to be the finite set of non-empty reduced words up to this length proves the claim.
\end{proof}

\subsection{The phase and its derivative}\label{sec:phaseder}
For all words $\Aa, \Bb \in \mathcal{W}^\circ$ and points $z \in D$, we define the \textit{phase} function
\begin{equation}\label{phase}
\Phi_{\Aa, \Bb}(z) := \overline{\mathbb{L}(\gamma_{\Bb}'(z))} -\mathbb{L}(\gamma_{\Aa}'(z)),
\end{equation}
where $\mathbb{L}$ is the complex logarithm as in \eqref{defi:principal_branch}. For all $s = \sigma + i t \in \CC$, this yields
\[
\gamma_{\Aa}'(z)^s \cdot \overline{\gamma_{\Bb}'(z)^s}
= \gamma_{\Aa}'(z)^{\sigma} \cdot \overline{\gamma_{\Bb}'(z)}^{\sigma} \cdot e^{-i t \Phi_{\Aa, \Bb}(z)}.
\]
Recall that for a word $\textbf{a}$ in $\mathcal{A} = \{1, \dots, 2m\}$, we let $\mathsf{red}(\textbf{a}) \in \mathcal{W}$ be its reduced form, obtained by applying the rule $a\overline{a} = \emptyset$ for all $a\in \mathcal{A}$. Recall also the definition of $\Upsilon_{\Aa}$ from \eqref{defi:Upsilon}, and define
\begin{equation}\label{Dab}
\mathcal{D}_{\Aa, \Bb} := \left( \frac{\Upsilon_{\Aa} \Upsilon_{\Bb}}{\Upsilon_{\mathsf{red}(\Aa \overline{\Bb})}} \right)^{1/2}.
\end{equation}
Consider the following cases:
\begin{itemize}
\item If $\Aa \to \overline{\Bb}$, then $\mathsf{red}(\Aa \overline{\Bb}) = \Aa \overline{\Bb}$ and $\Upsilon_{\Aa \overline{\Bb}} \asymp \Upsilon_{\Aa} \Upsilon_{\Bb}$, so $\mathcal{D}_{\Aa, \Bb} \asymp 1$.
\item If alternatively $\Aa \not\to \overline{\Bb}$, then $\Aa \overline{\Bb}$ is no longer reduced, and $\mathcal{D}_{\Aa, \Bb}$ may be very small.
\item In the extreme case $\Aa = \Bb$, we have $\mathsf{red}(\Aa \overline{\Bb}) = \emptyset$ and $\Upsilon_{\emptyset} = 1$, so $\mathcal{D}_{\Aa, \Bb} \asymp \Upsilon_\Aa$.
\item In any case, we have
\begin{equation}\label{eq:uniformDab}
\mathcal{D}_{\Aa, \Bb} = O(1),
\end{equation}
uniformly in $\Aa, \Bb$. Indeed, let $\Cc$ be the maximal common suffix of $\Aa$ and $\Bb$, and write $\Aa=\Ee\Cc$ and $\Bb=\Ff\Cc$. If $\Ee,\Ff\in\mathcal{W}^\circ$, then $\Ee\to\overline{\Ff}$ and $\mathsf{red}(\Aa\overline{\Bb})=\Ee\overline{\Ff}$. Repeated use of Lemma~\ref{lem:boundsderivatives} gives
\[
\frac{\Upsilon_{\Aa}\Upsilon_{\Bb}}{\Upsilon_{\mathsf{red}(\Aa\overline{\Bb})}}
\asymp \Upsilon_\Cc^2.
\]
The same estimate follows directly when one or both of $\Ee,\Ff$ are empty. Since $\Upsilon_\Cc\ll1$, this proves \eqref{eq:uniformDab}.
\end{itemize}

We now prove the key estimate:

\begin{prop}[Phase derivatives]\label{prop:phaseDerivative}
Let $b \in \mathcal{A}$ and $\Aa, \Bb \in \mathcal{W}^\circ$ with $\Aa, \Bb \to b$ and $\Aa \ne \Bb$. Then
\[
\inf_{x \in I_b} |\Phi_{\Aa, \Bb}'(x)| \asymp_\Gamma \mathcal{D}_{\Aa, \Bb}.
\]
\end{prop}

For the remainder of this section, we write for each word $\Aa$:
$$
\gamma_{\Aa} = \pmat{a_{\Aa}}{b_{\Aa}}{c_{\Aa}}{d_{\Aa}}\in \mathrm{SL}_2(\R).
$$

We need the following simple but crucial observation:

\begin{lemma}\label{lem:matrixEntriesBound}
For every word $\Aa$ over the alphabet $\mathcal{A}$ with $\mathsf{red}(\Aa) \neq \emptyset$ (equivalently, $\gamma_\Aa \neq \mathrm{id}$),
\[
|c_{\Aa}| \asymp_\Gamma \Upsilon_{\mathsf{red}(\Aa)}^{-1/2}.
\]
\end{lemma}

\begin{proof}
Since $\gamma_{\Aa} = \gamma_{\mathsf{red}(\Aa)}$, we may assume $\Aa = a_1 \cdots a_n$ is reduced and, by hypothesis, non-empty. Since all Schottky disks are bounded by assumption and do not contain $\infty$, we have $-\frac{d_{\Aa}}{c_{\Aa}} = \gamma_{\overline{\Aa}}(\infty)\neq \infty$, so $c_{\Aa} \neq 0$ for all non-empty reduced words $\Aa$.

For $b \in \mathcal{A}$ with $\Aa \to b$ and $z \in D_b$,
\[
\gamma_{\Aa}'(z) = \frac{1}{(c_{\Aa} z + d_{\Aa})^2}
= \frac{1}{c_{\Aa}^2 (z - x_{\Aa})^2}, \quad \text{where } x_{\Aa} = -\frac{d_{\Aa}}{c_{\Aa}}.
\]
As in the proof of Lemma~\ref{lem:taylor}, we have $|z - x_{\Aa}| \asymp_\Gamma 1$, so $|\gamma_{\Aa}'(z)| \asymp_\Gamma \frac{1}{c_{\Aa}^2}$. The claim now follows from Lemma~\ref{lem:boundsderivatives}.
\end{proof}

\begin{proof}[Proof of Proposition~\ref{prop:phaseDerivative}]
Fix $b \in \mathcal{A}$, $x \in I_b$, and distinct $\Aa, \Bb \in \mathcal{W}^\circ$ with $\Aa, \Bb \to b$. Direct calculation yields
\begin{equation}\label{firstDer}
\Phi_{\Aa, \Bb}'(x) = 2\left( \frac{c_{\Aa}}{c_{\Aa} x + d_{\Aa}} - \frac{c_{\Bb}}{c_{\Bb} x + d_{\Bb}}  \right)
= 2 \cdot \frac{c_{\Aa} d_{\Bb} - c_{\Bb} d_{\Aa}}{(c_{\Aa} x + d_{\Aa})(c_{\Bb} x + d_{\Bb})}.
\end{equation}
Therefore,
\[
|\Phi_{\Aa, \Bb}'(x)| = 2 |c_{\Aa} d_{\Bb} - c_{\Bb} d_{\Aa}| \cdot |\gamma_{\Aa}'(x)|^{1/2} |\gamma_{\Bb}'(x)|^{1/2}.
\]
By Lemma~\ref{lem:boundsderivatives},
\[
|\Phi_{\Aa, \Bb}'(x)| \asymp_\Gamma |c_{\Aa} d_{\Bb} - c_{\Bb} d_{\Aa}| \cdot \Upsilon_{\Aa}^{1/2} \Upsilon_{\Bb}^{1/2}.
\]
Note that
\[
\gamma_{\Aa \overline{\Bb}} = \gamma_{\Aa} \gamma_{\Bb}^{-1} =
\begin{pmatrix}
\ast & \ast \\
c_{\Aa} d_{\Bb} - c_{\Bb} d_{\Aa} & \ast
\end{pmatrix}.
\]
Since $\Aa$ and $\Bb$ are distinct reduced words, the bijection between reduced words and elements of $\Gamma$ gives $\gamma_\Aa \neq \gamma_\Bb$, so $\gamma_{\Aa\overline{\Bb}} = \gamma_\Aa \gamma_\Bb^{-1} \neq \mathrm{id}$, i.e. $\mathsf{red}(\Aa \overline{\Bb}) \neq \emptyset$. Hence Lemma~\ref{lem:matrixEntriesBound} applies to the word $\Aa\overline{\Bb}$, giving
\[
|c_{\Aa} d_{\Bb} - c_{\Bb} d_{\Aa}| \asymp_\Gamma \Upsilon_{\mathsf{red}(\Aa \overline{\Bb})}^{-1/2}.
\]
Therefore,
\[
|\Phi_{\Aa, \Bb}'(x)| \asymp_\Gamma \frac{\Upsilon_{\Aa}^{1/2} \Upsilon_{\Bb}^{1/2}}{\Upsilon_{\mathsf{red}(\Aa \overline{\Bb})}^{1/2}} = \mathcal{D}_{\Aa, \Bb},
\]
as claimed.
\end{proof}

\subsection{Averaged oscillatory integrals}\label{sec:averoscint}
Let $\varphi \in C^\infty(\mathbb{R})$ be a non-negative bump function with $\mathrm{supp}(\varphi) = [-2, 2]$ and $\varphi \equiv 1$ on $[-1, 1]$. Define
\[
\varphi_{T,H}(t) := \frac{1}{H}\, \varphi\left( \frac{t - T}{H} \right).
\]
Define the Fourier transform as usual by
\[
\widehat{\varphi}(\xi) := \int_{-\infty}^\infty \varphi(t)\, e^{-it\xi}\, dt.
\]
The aim of this section is to establish

\begin{prop}[Bound for averaged oscillatory integrals]\label{prop:FOURIER}
Let $\varphi_{T,H}$ be as above. There exists $T_0 = T_0(\Gamma) > 0$ such that for all $T \ge T_0$, $\eta > 0$, $T^\eta \le H \le T$, $b \in \mathcal{A}$, $\Aa, \Bb \in \mathcal{W}^\circ$ with $\Aa, \Bb \to b$, any measurable $f \colon D \to \CC$, and any $\epsilon > 0$, the following holds with $h := T^{-1}$:
\[
\left| \int_{\Omega_b(h)} \widehat{\varphi_{T,H}}\left(\Phi_{\Aa,\Bb}(z)\right) f(z)\, \mathrm{dvol}(z) \right| \ll_{\epsilon,\eta,\Gamma} 
h^{-\delta + 2} \|f\|_{\infty, \Omega_b(h)}.
\]
If, in addition, $\Aa \neq \Bb$, then
\[
\left| \int_{\Omega_b(h)} \widehat{\varphi_{T,H}}\left(\Phi_{\Aa,\Bb}(z)\right) f(z)\, \mathrm{dvol}(z) \right| \ll_{\epsilon,\eta,\Gamma}
h^{-\delta + 2} \mathcal{D}_{\Aa, \Bb}^{-\delta} H^{-\delta + \epsilon} \|f\|_{\infty, \Omega_b(h)}.
\]
Here, $\Omega_b(h)$ is as in \eqref{defi:omegab}, $\mathcal{D}_{\Aa, \Bb}$ is as in \eqref{Dab}, and
\[
\|f\|_{\infty, \Omega_b(h)} := \sup_{z \in \Omega_b(h)} |f(z)|.
\]
\end{prop}

Recall that $\Lambda(h)$ is the real $h$-neighbourhood of the limit set as defined in \eqref{defi:realhneighbourhood}. Before proving Proposition \ref{prop:FOURIER}, we need a bound on the size of $\Lambda(h)$ intersected with small intervals:

\begin{lemma}\label{lem:boundInSmallIntervals}
There exists a constant $C > 0$, depending only on $\Gamma$, such that for every $x_0 \in I = \bigsqcup_{a \in \mathcal{A}} I_a$ and all $\nu, h > 0$, we have
\[
|\Lambda(h) \cap [x_0 - \nu, x_0 + \nu]| \le C h \left(1 + \frac{\nu}{h} \right)^\delta,
\]
where $\vert \cdot\vert$ denotes the Lebesgue measure.
\end{lemma}

\begin{proof}
Write $I(p; h) := [p - h, p + h]$. Choose $h_0>0$ sufficiently small. If $h\ge h_0$, then compactness of $\Lambda$ gives $|\Lambda(h)|\ll h$, proving the result. We may thus assume that $0<h<h_0$. If $\Lambda(h)\cap I(x_0;\nu)$ is empty, there is nothing to prove.

The set $\Lambda(h) \cap I(x_0;\nu) $ is contained in the set of all points $x\in I(x_0;\nu)$ for which there exists some $p\in \Lambda$ such that $\vert x-p\vert \leq h$. In this case, the triangle inequality implies that $p\in I(x_0; \nu + h).$ It follows that the family of intervals $\{ I(p; h) : p \in \Lambda \cap I(x_0;  \nu + h) \}$ covers $\Lambda(h) \cap I(x_0; \nu)$. Observe that each of the intervals in this family satisfies
\begin{equation}\label{eq:intervalContainment}
I(p; h) \subset I(x_0; \nu + 2h).
\end{equation}
Choose a maximal disjoint subcollection $\{I(q_j; h)\}_{j=1}^m$ of this family. It is finite, since its centres are $2h$-separated in a bounded interval. By maximality, the enlarged intervals $\{I(q_j; 3h)\}_{j=1}^m$ cover $\Lambda(h) \cap I(x_0; \nu)$.

Let $\mu$ be the Patterson–Sullivan measure associated with $\Gamma$ (see \cite[Chapter~14]{Borthwick_book}). Since the intervals $I(q_j; h)$ are disjoint and each is contained in $I(x_0; \nu + 2h)$ by \eqref{eq:intervalContainment}, we have
\begin{equation}\label{eq:disjointness}
\sum_{j=1}^m \mu(I(q_j; h)) \le \mu(I(x_0; \nu + 2h)).
\end{equation}
The measure $\mu$ is Ahlfors regular of exponent $\delta$ (see \cite[Lemma~14.13]{Borthwick_book}): for all $p \in \Lambda$, we have
\begin{equation}\label{eq:measGrowth}
\mu(I(p; h)) \asymp h^\delta
\end{equation}
uniformly in $p$. The upper bound also holds when $p\notin \Lambda$, after possibly enlarging the implied constant. Hence,
\begin{equation}\label{eq:measGrowth2}
\mu(I(x_0; \nu + 2h)) \ll (\nu + h)^\delta.
\end{equation}
Combining \eqref{eq:measGrowth}, \eqref{eq:measGrowth2}, and \eqref{eq:disjointness}, we deduce that the number of disjoint intervals is bounded by
\[
m \ll \frac{(\nu + h)^\delta}{h^\delta} = \left(1 + \frac{\nu}{h}\right)^\delta.
\]
Since the intervals $\{I(q_j; 3h)\}_{j=1}^m$ cover $\Lambda(h) \cap I(x_0; \nu)$, we conclude
\[
|\Lambda(h) \cap I(x_0; \nu)| \le \sum_{j=1}^m |I(q_j; 3h)| = 3hm \ll h \left(1 + \frac{\nu}{h} \right)^\delta,
\]
as claimed.
\end{proof}

We are now ready to give the

\begin{proof}[Proof of Proposition~\ref{prop:FOURIER}]
By the scaling and translation properties of the Fourier transform, we have
\begin{equation}\label{eq:scaling}
\widehat{\varphi_{T,H}}(\xi) = e^{-i\xi T} \widehat{\varphi}(H\xi).
\end{equation}
Since $\varphi$ is smooth and supported in $[-2,2]$, repeated integration yields for all $m \in \mathbb{N}$ and $\xi \in \CC \setminus \{0\}$,
\[
|\widehat{\varphi}(\xi)| \ll_m \frac{e^{2 |\mathrm{Im}(\xi)|}}{|\xi|^m},
\]
which combined with \eqref{eq:scaling} gives
\begin{equation}\label{eq:fourierDecay}
|\widehat{\varphi_{T,H}}(\xi)| \ll_m \frac{e^{2(H+T)|\mathrm{Im}(\xi)|}}{(H |\xi|)^m}. 
\end{equation}
Let $z = x + iy \in \Omega_b(h)$. By Lemma~\ref{lem:taylor},
\[
\mathbb{L}(\gamma_{\Aa}'(z)) = \log(\gamma_{\Aa}'(x)) + O(h),
\]
with an implied constant depending only on $\Gamma$, and similarly for $\Bb$. Hence,
\begin{equation}\label{eq:phiBound}
\Phi_{\Aa,\Bb}(z) = \Phi_{\Aa,\Bb}(x) + O(h).
\end{equation}
In particular, since $\Phi_{\Aa,\Bb}(x)$ is real,
\begin{equation}\label{eq:imaginaryBound}
|\mathrm{Im}(\Phi_{\Aa,\Bb}(z))| = O(h). 
\end{equation}
Consequently, the triangle inequality, \eqref{eq:scaling}, and Lemma~\ref{lem:vol} give
\[
\left| \int_{\Omega_b(h)} \widehat{\varphi_{T,H}}(\Phi_{\Aa,\Bb}(z)) f(z)\, \dvol(z) \right|
\ll \vol(\Omega_b(h)) e^{O((T+H)h)} \|f\|_{\infty,\Omega_b(h)}
\ll h^{-\delta+2}\|f\|_{\infty,\Omega_b(h)}.
\]
This proves the first estimate. Assume henceforth that $\Aa \neq \Bb$.

Fix $\epsilon > 0$. Suppose first we have $|\Phi_{\Aa,\Bb}(z)| \ge H^{-1+\epsilon}$ for all $z \in \Omega_b(h)$. Then applying \eqref{eq:fourierDecay} together with \eqref{eq:imaginaryBound}, the conditions $T^\eta \le H \le T$ and $h= T^{-1}$, as well as Lemma~\ref{lem:vol}, we obtain
\begin{align*}
\left| \int_{\Omega_b(h)} \widehat{\varphi_{T,H}}(\Phi_{\Aa,\Bb}(z)) f(z)\, \dvol(z) \right| &\ll_m \frac{\vol(\Omega(h)) e^{O(h(T+H))}}{H^{\epsilon m}} \Vert f \Vert_{\infty, \Omega_b(h)}\\
&\ll T^{-2+\delta - \eta \epsilon m} \Vert f \Vert_{\infty, \Omega_b(h)}.
\end{align*}
Choosing $m$ large enough in terms of $\eta$ and $\epsilon$ implies the integral is $O_{A,\eta,\epsilon}(T^{-A} \Vert f \Vert_{\infty, \Omega_b(h)})$ for all $A > 0$.

Now suppose instead that there exists $z_0 = x_0 + i y_0 \in \Omega_b(h)$ with $|\Phi_{\Aa,\Bb}(z_0)| < H^{-1+\epsilon}$. From \eqref{eq:phiBound} we get for all $z = x + iy \in \Omega_b(h)$,
\[
|\Phi_{\Aa,\Bb}(z) - \Phi_{\Aa,\Bb}(z_0)| \ge |\Phi_{\Aa,\Bb}(x) - \Phi_{\Aa,\Bb}(x_0)| - C h
\]
for some $C>0$ depending only on $\Gamma$. By Proposition~\ref{prop:phaseDerivative},
\[
\inf_{I_b} |\Phi_{\Aa,\Bb}'| \asymp_\Gamma \mathcal{D}_{\Aa,\Bb}.
\]
Thus, applying the mean value theorem and noticing that $x, x_0 \in I_b$ yields
\[
|\Phi_{\Aa,\Bb}(z) - \Phi_{\Aa,\Bb}(z_0)| \ge c \mathcal{D}_{\Aa,\Bb} |x - x_0| - C h,
\]
for some constant $c > 0$ depending only on $\Gamma$. Clearly, this implies
\[
|\Phi_{\Aa,\Bb}(z)| \ge c \mathcal{D}_{\Aa,\Bb} |x - x_0| - C h - H^{-1+\epsilon}.
\]
Since $h = T^{-1} \le H^{-1}$, we now deduce that if
\[
|x - x_0| > c^{-1} (C + 2) \mathcal{D}_{\Aa,\Bb}^{-1} H^{-1+\epsilon} =: \nu,
\]
then
\[
|\Phi_{\Aa,\Bb}(z)| \ge H^{-1+\epsilon}.
\]
In view of this, we split
\[
\Omega_b(h) = \Omega_b^{(1)}(h) \sqcup \Omega_b^{(2)}(h),
\]
where
\begin{align*}
\Omega_b^{(1)}(h) &:= \{ z = x + iy \in \Omega_b(h) : |x - x_0| \le \nu \}, \\
\Omega_b^{(2)}(h) &:= \{ z = x + iy \in \Omega_b(h) : |x - x_0| > \nu \}.
\end{align*}
By construction, $z\in \Omega_b^{(2)}(h) $ implies $|\Phi_{\Aa,\Bb}(z)| \ge H^{-1+\epsilon}$. Hence, using \eqref{eq:fourierDecay}, Lemma~\ref{lem:vol}, and recalling that $T^\eta \le H \le T$ and $h= T^{-1}$, we obtain for all $m \in \mathbb{N}$,
\begin{align}
\left|\int_{\Omega_b(h)} \widehat{\varphi_{T,H}}(\Phi_{\Aa,\Bb}(z)) f(z)\, \dvol(z)\right|
&\le \left|\int_{\Omega_b^{(1)}(h)} \cdots\right| + \left|\int_{\Omega_b^{(2)}(h)} \cdots\right| \\
&\ll_m \vol(\Omega_b^{(1)}(h)) \|f\|_{\infty,\Omega_b(h)}
+ \frac{\vol(\Omega(h)) \|f\|_{\infty,\Omega_b(h)}}{H^{\epsilon m}} \\
&\ll \left( \vol(\Omega_b^{(1)}(h)) + T^{-2+\delta - \eta \epsilon m} \right) \|f\|_{\infty,\Omega_b(h)}. \label{eq:integralSplit}
\end{align}
It remains to bound the area of $\Omega_b^{(1)}(h)$. By Lemma~\ref{lem:structureLambdah}, each connected component of $\Omega_b(h)$ has diameter at most $C h$, so
\[
\Omega_b(h) \subseteq \Lambda(h) + i[-Ch, Ch]
\]
and therefore,
\[
\Omega_b^{(1)}(h) \subseteq (\Lambda(h) \cap [x_0 - \nu, x_0 + \nu]) + i[-Ch, Ch].
\]
From the bound in \eqref{eq:uniformDab} and $h = T^{-1} \le H^{-1}$, it follows that $\nu \gg h$. Thus, Lemma~\ref{lem:boundInSmallIntervals} gives
\begin{align*}
\vol(\Omega_b^{(1)}(h)) 
&\le 2Ch \cdot |\Lambda(h) \cap [x_0 - \nu, x_0 + \nu]| \\
&\ll h^2 \left(1 + \frac{\nu}{h}\right)^\delta \\
&\ll h^{-\delta + 2} \nu^\delta \\
&\ll h^{-\delta + 2} \mathcal{D}_{\Aa,\Bb}^{-\delta} H^{-\delta(1 - \epsilon)}.
\end{align*}
Since $0<\delta<1$ and $H\ge 1$, we have $H^{-\delta(1-\epsilon)}\le H^{-\delta+\epsilon}$. Substituting this into \eqref{eq:integralSplit} and choosing $m$ sufficiently large in terms of $\epsilon$ and $\eta$, the term $T^{-2+\delta-\eta\epsilon m}$ is absorbed by the same bound. This completes the proof.
\end{proof}

\subsection{Applying Jensen's Formula}\label{sec:jensen}
Recall the definition of the $\tau$-refined transfer operators $\mathcal{L}_{\tau, s}$ from \eqref{defiRefinedTO} and the refined domain $\Omega(h)$ from \eqref{defi:omega(h)}. If $h > 0$ and $\tau > 0$ are sufficiently small in terms of the Schottky data of $\Gamma$, then $\mathcal{L}_{\tau, s} \colon H^2(\Omega(h)) \to H^2(\Omega(h))$ is well-defined and trace class. We consider the concatenated operator
\begin{equation}\label{defConcat}
\mathcal{L}_{\tau_0, \tau_1, s} := \mathcal{L}_{\tau_1, s} \mathcal{L}_{\tau_0, s} \colon H^2(\Omega(h)) \to H^2(\Omega(h)).
\end{equation}

The Hilbert–Schmidt norm of a trace class operator $A \colon H \to H$ on a separable Hilbert space $H$ is given by
\[
\|A\|_{\mathrm{HS}}^2 := \mathrm{tr}(A^\ast A),
\]
where $A^\ast$ denotes the adjoint of $A$. Recall also the resonance counting function from \eqref{defi:NXTH}:
\[
N_X(\sigma, T, H) := \#\left\{ s \in \mathcal{R}_X : \mathrm{Re}(s) \ge \sigma,\ \mathrm{Im}(s) \in [T - H, T + H] \right\},
\]
where resonances are counted with multiplicities.

\begin{prop}[Resonance counting bound via HS-norm]\label{prop:intHSzeros}
There exist positive constants $\alpha$, $\beta$, $\epsilon_0$, $K_0$, $T_0$, and $C$, depending only on $\Gamma$, such that for all $0<\sigma<\delta$, $T > T_0$, $1 \le H \le T$, $K > K_0$, and $h, \tau_0, \tau_1 \in (0, \epsilon_0)$, we have
\[
N_X(\sigma, T, H) \le C K^2 \max_{\substack{ \sigma - \frac{\alpha}{K} \le \mathrm{Re}(s) \le \beta K \\ |\mathrm{Im}(s)| \le \beta K }} \left( \int_{T - H}^{T + H} \| \mathcal{L}_{\tau_0, \tau_1, s + it} \|_{\mathrm{HS}, h}^2 dt \right) + K^2 (C \tau_0 \tau_1)^{2K} T,
\]
where $\|\cdot\|_{\mathrm{HS}, h}$ denotes the Hilbert--Schmidt norm on $H^2(\Omega(h))$.
\end{prop}

This proposition follows from the following variant of Jensen's formula that we specifically tailored to our purposes:

\begin{lemma}[Adapted Jensen’s formula]\label{lem:jensenAdapt}
Let $0<\sigma<\delta$, $T\in\mathbb{R}$, $H\ge1$, and let $f$ be a nonzero entire function. Define
\[
N_f(\sigma, T, H) := \#\left\{ s \in \CC : f(s) = 0,\ \sigma \le \mathrm{Re}(s) \le \delta,\ \mathrm{Im}(s) \in [T - H, T + H] \right\}.
\]
Then for all $K$ sufficiently large,
\[
N_f(\sigma, T, H) \ll K^2 \left( \max_{\substack{ \sigma - \frac{\alpha}{K} \le \mathrm{Re}(s) \le \beta K \\ |\mathrm{Im}(s)| \le \beta K }} \int_{T - H}^{T + H} \log |f(s + it)| dt - \int_{T - H}^{T + H} \log |f(\delta + K + it)| dt \right),
\]
where the implied constants as well as $\alpha, \beta > 0$ are independent of $f, \sigma$, and $K$.
\end{lemma}

\begin{proof}
Fix $t \in \mathbb{R}$ and let $D_1 = D_\CC(s_0, r_1)$ and $D_2 = D_\CC(s_0, r_2)$ be concentric disks centered at $s_0 = \sigma_0 + it$ with radii $r_2 > r_1 > 0$, chosen such that
\begin{equation}\label{concentric_jensen}
\left\{ s \in \CC : \sigma \le \mathrm{Re}(s) \le \delta,\ |\mathrm{Im}(s) - t| \le 1 \right\} \subset \overline{D_1} \subset D_2.
\end{equation}
Define the zero counting function
\[
M_f(\sigma, t) := \#\left\{ s \in \CC : f(s) = 0,\ \sigma \le \mathrm{Re}(s) \le \delta,\ |\mathrm{Im}(s) - t| \le 1 \right\},
\]
where zeros are counted with multiplicities. By classical Jensen's formula (see for instance \cite{Titchmarsh}),
\[
M_f(\sigma, t) \le \frac{1}{\log(r_2 / r_1)} \left( \int_0^1 \log |f(\sigma_0 + r_2 e^{2\pi i \theta} + it)| d\theta - \log |f(\sigma_0 + it)| \right).
\]
(If $f(\sigma_0+it) = 0$ for some $t$, the right-hand side equals $+\infty$ and the inequality holds trivially, so this causes no issue upon integrating over $t$.) Integrating this over $t \in [T - H, T + H]$ gives
\begin{align*}
N_f(\sigma, T, H) &\le \int_{T - H}^{T + H} M_f(\sigma, t)\, dt \\
&\le \frac{1}{\log(r_2 / r_1)} \left( \int_0^1 \int_{T - H}^{T + H} \log |f(\sigma_0 + r_2 e^{2\pi i \theta} + it)| dt\, d\theta - \int_{T - H}^{T + H} \log |f(\sigma_0 + it)| dt \right).
\end{align*}
Now let $K\gg 1 $ and choose parameters
\[
\sigma_0 := \delta + K, \quad r_1 := \sqrt{(\sigma_0-\sigma)^2 + 1}, \quad r_2 := r_1 + \frac{1}{K}.
\]
These choices guarantee that \eqref{concentric_jensen} holds true. For large $K$, we have
\begin{itemize}
\item $r_1 \asymp r_2 \asymp \sigma_0 - \sigma \asymp K$;
\item $\bigl(\log(r_2 / r_1)\bigr)^{-1} \ll K^2$;
\item $r_1 = \sigma_0 - \sigma + O(\frac{1}{K}) $ and $r_2 = \sigma_0 - \sigma + O(\frac{1}{K}) $;
\item $\sigma - O(1/K) \le\mathrm{Re}(\sigma_0 + r_2 e^{2\pi i \theta}) \le O(K)$ for all $\theta\in [0,1)$;
\item $|\mathrm{Im}(\sigma_0 + r_2 e^{2\pi i \theta})| \le \sigma_0 + r_2 = O(K)$ for all $\theta\in [0,1)$.
\end{itemize}
Thus,
\[
\int_0^1 \int_{T - H}^{T + H} \log |f(\sigma_0 + r_2 e^{2\pi i \theta} + it)| dt\, d\theta
\le \max_{\substack{ \sigma - O(1/K) \le \mathrm{Re}(s) \le O(K) \\ |\mathrm{Im}(s)| \le O(K) }} \int_{T - H}^{T + H} \log |f(s + it)| dt.
\]
Combining all estimates yields the stated bound.
\end{proof}

We also require the following:

\begin{lemma}[Pointwise estimate for $\mathrm{Re}(s) \geq \delta + 1$]\label{lem:pointwiseEst} For all sufficiently small resolution parameters $\tau_0 > 0 $ and $\tau_1  >0 $, and for all $s = \sigma +it\in \CC$ with $\sigma \ge \delta+1,$ the Fredholm determinant of $ \mathcal{L}_{\tau_0, \tau_1, s}^2 $ satisfies
\begin{equation*}
- \log \vert \det\left( 1-\mathcal{L}_{\tau_0, \tau_1, s}^2  \right)\vert \le (C \tau_0 \tau_1)^{2(\sigma-\delta)},
\end{equation*}
where $C>0$ depends only on $\Gamma$.
\end{lemma}

\begin{proof}
We will adapt the argument of Magee--Naud \cite{NaudMagee}. Let $H$ be a separable Hilbert space and $A\colon H\to H$ a trace class operator. Then the Fredholm determinant of $A$ can be expressed as
\begin{equation}
\det(1-A) = \exp\left( -\sum_{n=1}^\infty \frac{1}{n} \mathrm{tr}(A^n) \right),
\end{equation}
whenever the series converges absolutely (see, for instance, \cite{Gohberg_Goldberg_Krupnik}). In this case, taking absolute values and logarithms on both sides gives
\begin{equation}
- \log \vert \det(1-A)\vert =  \sum_{n=1}^\infty \frac{1}{n} \mathrm{Re}(\mathrm{tr}(A^n)) \le \sum_{n=1}^\infty \frac{1}{n} \vert \mathrm{tr}(A^n) \vert.
\end{equation}
Applying this to $A=\mathcal{L}_{\tau_0, \tau_1, s}^2$ yields
\begin{equation}\label{ineq_exp}
- \log \vert \det( 1- \mathcal{L}_{\tau_0, \tau_1, s}^2)\vert \le  \sum_{n=1}^\infty \frac{1}{n} \vert \mathrm{tr}(\mathcal{L}_{\tau_0, \tau_1, s}^{2n})\vert.
\end{equation}
We need a suitable upper bound for the trace of $\mathcal{L}_{\tau_0, \tau_1, s}^{2n}$. By \eqref{defiRefinedTO} we can write
\begin{equation}
\mathcal{L}_{\tau_0, \tau_1, s} f(z) = \sum_{(\Bb_0,\Bb_1)\in S_b} \gamma'_{\Bb_0\Bb_1}(z)^s  f(\gamma_{\Bb_0\Bb_1}(z)) \quad \text{for}\quad z\in \Omega_b(h),
\end{equation}
where, for each $b\in \mathcal{A}$, we define
$$
S_b = \{ (\Aa_0,\Aa_1) : \Aa_1\in Y_b(\tau_1),\ \Aa_0\in Y_{\mathrm{first}(\Aa_1)}(\tau_0) \}.
$$
Let $S:=\bigsqcup_{b\in\mathcal{A}}S_b$. Every element in $S$ corresponds to a triple $\beta=(b,\Bb_0,\Bb_1)$. Set $\Bb(\beta):=\Bb_0\Bb_1$. A sequence $\beta_j=(b_j,\Bb_{j,0},\Bb_{j,1})$, $1\le j\le n$, is cyclically composable if
$$
b_j=\mathrm{first}(\Bb(\beta_{j+1})) \quad (1\le j\le n), \qquad \beta_{n+1}:=\beta_1.
$$
For such a sequence, put $\Aa:=\Bb(\beta_1)\cdots\Bb(\beta_n)$ and let $x_{\Aa}$ be the unique attracting fixed point of $\gamma_{\Aa}$, lying in $\Omega_{b_n}(h)$. The Lefschetz fixed point formula (see \cite[Lemma~15.9]{Borthwick_book}) gives
\begin{equation}\label{traceformulaS}
\mathrm{tr}(\mathcal{L}_{\tau_0,\tau_1,s}^{n})
=\sum_{\substack{\Aa=\Bb(\beta_1)\cdots\Bb(\beta_n) \\ \beta_1,\ldots,\beta_n\in S\\
b_j=\mathrm{first}(\Bb(\beta_{j+1})),\ 1\le j\le n}}
\frac{\gamma_{\Aa}'(x_{\Aa})^s}{1-\gamma_{\Aa}'(x_{\Aa})},
\end{equation}
where the sum is over all cyclically composable sequences $\beta_1,\ldots,\beta_n\in S$. By Lemma~\ref{lem:YZ} we have $\Upsilon_{\Bb_{j,0}}\ll\tau_0$ and $\Upsilon_{\Bb_{j,1}}\ll\tau_1$. Therefore, applying Lemma~\ref{lem:boundsderivatives} $(n-1)$ times, we obtain the bound
$$
0 < \gamma_{\Aa}'(x_{\Aa})
\ll \Upsilon_{\Aa}
\le C^n\prod_{j=1}^n\Upsilon_{\Bb(\beta_j)}
\le (C\tau_0\tau_1)^n,
$$
for some constant $C=C(\Gamma)>0$. For $C\tau_0\tau_1<1/2$, we now get
\begin{equation}\label{trace_bound_0}
|\mathrm{tr}(\mathcal{L}_{\tau_0, \tau_1, s}^{n})| \le 2\vert S\vert^n (C \tau_0 \tau_1)^{\sigma n}.
\end{equation}
By Lemma \ref{lem:YZ} the cardinality of $S$ is bounded by
\begin{equation}
\vert S\vert \le \vert Y(\tau_0)\vert \vert Y(\tau_1)\vert \le C ( \tau_0 \tau_1 )^{-\delta},
\end{equation}
after increasing $C$ if needed. Hence, we get
$$
|\mathrm{tr}(\mathcal{L}_{\tau_0, \tau_1, s}^{n})| \le 2 C^n ( \tau_0 \tau_1 )^{-n\delta} (C \tau_0 \tau_1)^{\sigma n} = 2 C^{(1+\delta)n} ( C\tau_0 \tau_1 )^{n(\sigma-\delta)}.
$$
Since we assume that $\sigma - \delta\ge 1 $, this gives
$$
|\mathrm{tr}(\mathcal{L}_{\tau_0, \tau_1, s}^{n})|  \le 2 (  C^{\frac{(1+\delta)}{\sigma-\delta}+1} \tau_0 \tau_1 )^{n(\sigma-\delta)} \le (  \tilde{C} \tau_0 \tau_1 )^{n(\sigma-\delta)}
$$
for some large enough constant $\tilde{C}>0.$ Returning to \eqref{ineq_exp}, taking $\tilde{C}\tau_0\tau_1 < 1/2$, we finally arrive at
$$
- \log \vert \det( 1- \mathcal{L}_{\tau_0, \tau_1, s}^2)\vert \le  \sum_{n=1}^\infty \vert \mathrm{tr}(\mathcal{L}_{\tau_0, \tau_1, s}^{2n})\vert \le \sum_{n=1}^\infty (\tilde{C} \tau_0 \tau_1)^{2n(\sigma-\delta) } \le 2(\tilde{C} \tau_0 \tau_1)^{2(\sigma-\delta)},
$$
completing the proof.
\end{proof}

We are now ready to finish the

\begin{proof}[Proof of Proposition \ref{prop:intHSzeros}]
Consider the entire function
$$
f(s) := \det( 1- \mathcal{L}_{\tau_0, \tau_1, s}^2)
$$
and observe that
$$
f(s) = \det( 1 + \mathcal{L}_{\tau_0, \tau_1, s}) \det( 1- \mathcal{L}_{\tau_0, \tau_1, s}).
$$
By Lemma \ref{lem:holomMultiple}, this function is a holomorphic multiple of $Z_\Gamma(s)$. It follows that resonances for $X$ (counted according to multiplicity) occur as zeros of $f(s)$. Combining Lemma \ref{lem:jensenAdapt} and Lemma \ref{lem:pointwiseEst}, we see that there are positive constants $\alpha,\beta,\epsilon_0,K_0,T_0,C$, depending only on $\Gamma$, such that for all $T > T_0$, $1\le H \le T$, $K > K_0$, and $h,\tau_0,\tau_1 \in (0,\epsilon_0)$ we have
\begin{equation}\label{ctnzerobeforHS}
N_X(\sigma, T, H) \le C K^2 \max_{\substack{ \sigma-\frac{\alpha}{K} \le  \mathrm{Re}(s)\le  \beta K \\ \vert \mathrm{Im}(s)\vert \le  \beta K }} \left( \int_{T-H}^{T+H} \log \vert  \det( 1- \mathcal{L}_{\tau_0, \tau_1, s+it}^2) \vert dt \right) +  K^2 T (C\tau_0 \tau_1)^{2K}.
\end{equation}
It remains to prove
$$
\log\vert \det( 1- \mathcal{L}_{\tau_0, \tau_1, s}^2) \vert \le  \Vert \mathcal{L}_{\tau_0, \tau_1, s} \Vert_{\mathrm{HS},h}^2.
$$
To that effect, we recall some basic facts on trace class operators and Fredholm determinants, referring the reader to \cite{Gohberg_Goldberg_Krupnik, Gohberg_Krein, Simon} for more details. Weyl's estimate on Fredholm determinants states that for every trace class operator $A\colon H \to H$ on a separable Hilbert space $H$, we have
\begin{equation}\label{WeylEstgeneral}
\log\vert \det\left( 1- A \right) \vert \le  \Vert A\Vert_{\mathrm{tr}},
\end{equation}
where $\Vert \cdot\Vert_{\mathrm{tr}}$ is the trace norm. Moreover, for any two Hilbert--Schmidt operators $A_1,A_2 \colon H \to H$ we have the Cauchy--Schwarz-type inequality 
\begin{equation}\label{traceHS}
\Vert A_1 A_2\Vert_{\mathrm{tr}} \le  \Vert A_1 \Vert_{\mathrm{HS}} \Vert A_2\Vert_{\mathrm{HS}}. 
\end{equation}
By Lemma \ref{lem:welldefinedness}, we know that if $h,\tau_0,\tau_1$ are sufficiently small,
$$
\mathcal{L}_{\tau_0,\tau_1,s} \colon H^2(\Omega(h)) \to H^2(\Omega(h)) 
$$
is a well-defined family of trace class operators. We can therefore apply the above facts to $A=A_1=A_2 = \mathcal{L}_{\tau_0, \tau_1, s}$ and $H=H^2(\Omega(h))$ to obtain
$$
\log\vert \det( 1- \mathcal{L}_{\tau_0, \tau_1, s}^2) \vert \le   \Vert \mathcal{L}_{\tau_0, \tau_1, s}^2 \Vert_{\mathrm{tr}} \le  \Vert \mathcal{L}_{\tau_0, \tau_1, s} \Vert_{\mathrm{HS},h}^2,
$$
as desired.
\end{proof}

\subsection{Hilbert--Schmidt norm}\label{sec:HSnorm}
The goal of this section is to prove the following:

\begin{prop}[HS-norm of $\mathcal{L}_{\tau_0, \tau_1, s}$]\label{prop:HSnorm}
Let $h,\tau_0, \tau_1 > 0$ be sufficiently small in terms of $\Gamma$. The Hilbert--Schmidt norm of the operator
\[
\mathcal{L}_{\tau_0, \tau_1, s} = \mathcal{L}_{\tau_1, s} \mathcal{L}_{\tau_0, s}\colon H^2(\Omega(h)) \to H^2(\Omega(h))
\]
is given by
\[
\Vert \mathcal{L}_{\tau_0, \tau_1, s} \Vert_{\mathrm{HS},h}^2 = \sum_{b\in \mathcal{A}}  
\sum_{(\Aa_0,\Aa_1),(\Bb_0,\Bb_1) \in S_b}
\int_{\Omega_b(h)} \gamma'_{\Aa_0\Aa_1}(z)^s \overline{\gamma'_{\Bb_0\Bb_1}(z)^s}  
B_{\Omega(h)}(\gamma_{\Aa_0\Aa_1}(z), \gamma_{\Bb_0\Bb_1}(z)) \dvol(z).
\]
Here
\[
S_b:=\{(\Aa_0,\Aa_1) : \Aa_1\in Y_b(\tau_1),\ \Aa_0\in Y_{\mathrm{first}(\Aa_1)}(\tau_0)\},
\]
and $B_{\Omega(h)}(\cdot,\cdot)$ denotes the Bergman reproducing kernel of $H^2(\Omega(h))$.
\end{prop}

\begin{proof}
For analogous formulas, see \cite[Lemma 4.7]{NaudMagee} and \cite[Prop.~5.5]{Pohl_Soares}. We give a more direct proof. First, observe that
\[
\mathcal{L}_{\tau_0, \tau_1, s}f(z) = \sum_{(\Aa_0,\Aa_1)\in S_b} \gamma'_{\Aa_0\Aa_1}(z)^s f(\gamma_{\Aa_0\Aa_1}(z)) \quad \text{for } z\in \Omega_b(h).
\]
By the Bergman kernel’s defining property (see Section \ref{sec:Bergman}), we have
\[
\int_{\Omega(h)} B_{\Omega(h)}(z,w)f(w)\dvol(w) = f(z).
\]
Thus, we can rewrite $\mathcal{L}_{\tau_0, \tau_1, s}$ as an integral kernel operator:
\[
\mathcal{L}_{\tau_0, \tau_1, s}f(z) = \int_{\Omega(h)} K(z,w)f(w)\dvol(w),
\]
with kernel
\[
K(z,w) = \sum_{(\Aa_0,\Aa_1)\in S_b} \gamma'_{\Aa_0\Aa_1}(z)^s B_{\Omega(h)}(\gamma_{\Aa_0\Aa_1}(z), w) \quad \text{for } z \in \Omega_b(h).
\]
The Hilbert--Schmidt norm can then be computed as follows:
\begin{align*}
\Vert \mathcal{L}_{\tau_0, \tau_1, s} \Vert_{\mathrm{HS}}^2 
&= \int_{\Omega(h)} \int_{\Omega(h)} |K(z,w)|^2 \dvol(w)\dvol(z) \\
&= \sum_{b\in \mathcal{A}} \int_{\Omega_b(h)} \int_{\Omega(h)} |K(z,w)|^2 \dvol(w)\dvol(z) \\
&= \sum_{b\in \mathcal{A}} \int_{\Omega_b(h)} \int_{\Omega(h)}
\sum_{(\Aa_0,\Aa_1),(\Bb_0,\Bb_1) \in S_b}
\gamma'_{\Aa_0\Aa_1}(z)^s
\overline{\gamma'_{\Bb_0\Bb_1}(z)^s} \\
&\qquad\qquad\qquad\qquad\qquad\qquad\cdot
B_{\Omega(h)}(\gamma_{\Aa_0\Aa_1}(z), w)
\overline{B_{\Omega(h)}(\gamma_{\Bb_0\Bb_1}(z), w)}
\dvol(w)\dvol(z).
\end{align*}
Swapping the integral and sum gives
\begin{align*}
\Vert \mathcal{L}_{\tau_0, \tau_1, s} \Vert_{\mathrm{HS}}^2
&= \sum_{b\in \mathcal{A}} \int_{\Omega_b(h)}
\sum_{(\Aa_0,\Aa_1),(\Bb_0,\Bb_1) \in S_b}
\gamma'_{\Aa_0\Aa_1}(z)^s
\overline{\gamma'_{\Bb_0\Bb_1}(z)^s} \\
&\qquad\qquad\qquad\qquad\qquad\cdot
\left( \int_{\Omega(h)}
B_{\Omega(h)}(\gamma_{\Aa_0\Aa_1}(z), w)
\overline{B_{\Omega(h)}(\gamma_{\Bb_0\Bb_1}(z), w)}
\dvol(w) \right) \dvol(z).
\end{align*}
Using the reproducing property, we obtain
\begin{align*}
\int_{\Omega(h)}
B_{\Omega(h)}(\gamma_{\Aa_0\Aa_1}(z), w)
\overline{B_{\Omega(h)}(\gamma_{\Bb_0\Bb_1}(z), w)}
\dvol(w)
&= \int_{\Omega(h)}
B_{\Omega(h)}(\gamma_{\Aa_0\Aa_1}(z), w)
B_{\Omega(h)}(w, \gamma_{\Bb_0\Bb_1}(z))
\dvol(w)\\
&= B_{\Omega(h)}(\gamma_{\Aa_0\Aa_1}(z), \gamma_{\Bb_0\Bb_1}(z)),
\end{align*}
and the result follows.
\end{proof}

\subsection{A bound for special sums over words}\label{sec:certainSum}
In the final steps of the proof of Theorem \ref{main_thm}, we will encounter sums of the form
\begin{equation}\label{eq:weird_sum}
\sum_{\Aa\in\mathcal{U},\,\Bb\in\mathcal{V}} \Upsilon_{\mathsf{red}(\Aa\overline{\Bb})}^{\alpha},
\end{equation}
where $\mathcal{U},\mathcal{V}\subset\mathcal{W}$ are finite sets satisfying $\Upsilon_{\Aa}, \Upsilon_{\Bb}\asymp \tau$ for all $\Aa\in\mathcal{U},\,\Bb\in\mathcal{V}$. The goal of this subsection is to prove the following estimate.

\begin{lemma}\label{lem:weird_sum}
Let $C_0\ge1$, and let $\mathcal{U},\mathcal{V}\subset\mathcal{W}$ be finite sets such that for all $\Aa\in\mathcal{U}$ and $\Bb\in\mathcal{V}$, we have
\[
C_0^{-1}\tau\le\Upsilon_{\Aa}\le C_0\tau,
\qquad
C_0^{-1}\tau\le\Upsilon_{\Bb}\le C_0\tau.
\]
Then, for all $0<\tau\le1$, $\alpha\in\mathbb{R}$, and $\epsilon>0$,
\[
\sum_{\Aa\in\mathcal{U},\,\Bb\in\mathcal{V}}
\Upsilon_{\mathsf{red}(\Aa\overline{\Bb})}^{\alpha}
\ll_{\alpha,\epsilon,C_0,\Gamma}
\tau^{-\delta}+\tau^{-2\delta+2\alpha-\epsilon}.
\]
\end{lemma}

\begin{remark}\label{rmk:trivial}
This bound is sharp up to the $\epsilon$-loss. Indeed, taking $\mathcal{U} = \mathcal{V} = Y(\tau),$ it follows from Lemma~\ref{lem:YZ} that
\[
\eqref{eq:weird_sum} \ge \sum_{\Aa \in Y(\tau)} 1 = |Y(\tau)| \gg \tau^{-\delta}.
\]
\end{remark}

\begin{proof}
We begin by proving that for all $\beta \in \mathbb{R}$, we have
\begin{equation}\label{eq:beta_bound}
\sum_{\substack{\Aa \in \mathcal{W}\\ \Upsilon_{\Aa} > \tau}} \Upsilon_{\Aa}^\beta \ll_{\beta,\epsilon,\Gamma} \tau^{-\epsilon} + \tau^{-\delta+\beta} .
\end{equation}
Note that $\Upsilon_{\Aa}$ is uniformly bounded from above in terms of $\Gamma$, so there exists some constant $C = C(\Gamma) > 0$ such that
\begin{align*}
\sum_{\substack{\Aa \in \mathcal{W}\\ \Upsilon_{\Aa} > \tau}} \Upsilon_{\Aa}^\beta 
&\le \sum_{\substack{k\in\mathbb{Z}\\ -C \le k \le \log_2(\tau^{-1})}} \left( \sum_{2^{-k} \ge \Upsilon_{\Aa} > 2^{-(k+1)}} \Upsilon_{\Aa}^\beta\right)\\ 
&\ll_{\beta} \sum_{\substack{k\in\mathbb{Z}\\ -C \le k \le \log_2(\tau^{-1})}} 2^{-\beta k} \#\left\{ \Aa \in \mathcal{W} : \Upsilon_{\Aa} > 2^{-(k+1)} \right\}.
\end{align*}
Recall from Lemma~\ref{lem:cardsets} that
$$
\#\left\{ \Aa \in \mathcal{W} : \Upsilon_{\Aa} > 2^{-(k+1)} \right\} \ll_{\Gamma} 2^{\delta k}.
$$
Thus,
$$
\sum_{\substack{\Aa \in \mathcal{W}\\ \Upsilon_{\Aa} > \tau}} \Upsilon_{\Aa}^\beta 
\ll_{\beta, \Gamma} \sum_{-C \le k \le \log_2(\tau^{-1})} 2^{(\delta-\beta) k}.
$$
Summing the resulting geometric series, handling separately the cases $\beta > \delta$ and $\beta \le \delta$, yields the bound in \eqref{eq:beta_bound}.

For $\Aa\in\mathcal{U}$ and $\Bb\in\mathcal{V}$, let $\Cc$ be their maximal common suffix and write $\Aa=\Aa_0\Cc$ and $\Bb=\Bb_0\Cc$. Then $\Aa_0,\Bb_0\to\Cc$ and $\Aa_0\to\overline{\Bb_0}$. Lemma~\ref{lem:boundsderivatives}, with the empty-word cases interpreted using $\Upsilon_{\emptyset}=1$, gives
\[
\Upsilon_{\mathsf{red}(\Aa\overline{\Bb})} = \Upsilon_{\Aa_0\overline{\Bb_0}} \asymp \Upsilon_{\Aa_0} \Upsilon_{\Bb_0},
\]
and, for a constant depending on $C_0$ and $\Gamma$,
\[
C^{-1} \Upsilon_{\Cc}^{-1} \tau < \Upsilon_{\Aa_0}, \Upsilon_{\Bb_0} < C \Upsilon_{\Cc}^{-1} \tau, \qquad \Upsilon_{\Cc} > C^{-1} \tau.
\]
It follows that
\begin{align*}
\sum_{\Aa\in\mathcal{U},\,\Bb\in\mathcal{V}} \Upsilon_{\mathsf{red}(\Aa\overline{\Bb})}^{\alpha}
&\ll_{\alpha,C_0,\Gamma} \sum_{\substack{\Aa_0,\Bb_0,\Cc \in \mathcal{W} \\ C^{-1} \Upsilon_{\Cc}^{-1} \tau < \Upsilon_{\Aa_0}, \Upsilon_{\Bb_0} < C \Upsilon_{\Cc}^{-1} \tau \\ \Upsilon_{\Cc} > C^{-1} \tau}} \Upsilon_{\Aa_0}^{\alpha} \Upsilon_{\Bb_0}^{\alpha} \\
&\ll_{\alpha,C_0,\Gamma} \tau^{2\alpha} \cdot \sum_{\substack{\Aa_0,\Bb_0,\Cc \in \mathcal{W} \\ \Upsilon_{\Aa_0}, \Upsilon_{\Bb_0} > C^{-1} \Upsilon_{\Cc}^{-1} \tau \\ \Upsilon_{\Cc} > C^{-1} \tau}} \Upsilon_{\Cc}^{-2\alpha} \\
&\ll_{\alpha,C_0,\Gamma} \tau^{2\alpha} \cdot \sum_{\substack{\Cc\in \mathcal{W} \\ \Upsilon_{\Cc} > C^{-1} \tau}} \#\left\{ (\Aa_0,\Bb_0) \in \mathcal{W}^2 : \Upsilon_{\Aa_0}, \Upsilon_{\Bb_0} > C^{-1} \Upsilon_{\Cc}^{-1} \tau \right\} \cdot \Upsilon_{\Cc}^{-2\alpha}.
\end{align*}
Since $\Upsilon_{\Cc}\gg\tau$, Lemma~\ref{lem:cardsets} bounds the cardinality in the last line by $O_\Gamma(\tau^{-2\delta}\Upsilon_{\Cc}^{2\delta})$. Therefore,
$$
\sum_{\Aa\in\mathcal{U},\,\Bb\in\mathcal{V}} \Upsilon_{\mathsf{red}(\Aa\overline{\Bb})}^{\alpha}
\ll_{\alpha,C_0,\Gamma} \tau^{-2\delta + 2\alpha} \left(  \sum_{\substack{\Cc \in \mathcal{W}\\ \Upsilon_{\Cc} > C^{-1} \tau}} \Upsilon_{\Cc}^{2\delta - 2\alpha}\right).
$$
Using the bound in \eqref{eq:beta_bound} with $\beta = 2\delta - 2\alpha$ on the right-hand side now yields the desired estimate.
\end{proof}

\subsection{Finishing the proof}\label{sec:finish}
We finally finish the proof of our main Theorem \ref{main_thm}. It will quickly follow from

\begin{prop}[Main Technical Estimate]\label{prop:main_estimate}
There exists $T_0 = T_0(\Gamma) > 0$ such that for all $T \ge T_0$, $\eta > 0$, $T^{\eta} \le H \le T$, $\sigma \ge 0$, $\epsilon > 0$, and $s \in \CC$ with $\mathrm{Re}(s) \ge \sigma$, $|\mathrm{Im}(s)| \le T$, the following holds: there are parameters $h, \tau_0, \tau_1$ such that
\[
\frac{1}{2H} \int_{T-H}^{T+H} \Vert \mathcal{L}_{\tau_0, \tau_1, s+it} \Vert_{\mathrm{HS}, h}^2 \, dt 
\ll_{\epsilon,\eta ,\Gamma} C^{\sigma} H^{-2\left(\sigma-\frac{\delta}{2}\right)+\epsilon} T^{2\delta - 2\sigma}.
\]
Moreover, this bound is achieved by choosing $h = T^{-1}$ and resolution parameters $\tau_0 = c_0 T^{-1}$ and $\tau_1 = c_1 H^{-1}$ with constants $c_0, c_1 > 0$ depending only on $\Gamma$.
\end{prop}

Before proving this proposition, let us use it to finish the proof of Theorem \ref{main_thm}. From Proposition \ref{prop:intHSzeros}, for sufficiently large $K$, we have
$$
N_X(\sigma, T, H) \le C K^2 \max_{\substack{ \sigma-\frac{\alpha}{K} \le  \mathrm{Re}(s)\le  \beta K \\ \vert \mathrm{Im}(s)\vert \le  \beta K }} \left( \int_{T-H}^{T+H} \Vert \mathcal{L}_{\tau_0, \tau_1, s+it} \Vert_{\mathrm{HS},h}^2 dt \right) +   K^2 T (C\tau_0 \tau_1)^{2K},
$$
where $C>0$ is some constant depending only on $\Gamma.$ We now apply Proposition \ref{prop:main_estimate} with the choices in that statement:
\[
h = T^{-1}, \quad \tau_0 = c_0 T^{-1}, \quad \tau_1 = c_1 H^{-1}.
\]
We get for any $0<\eta\le1$, $T \ge T_0 = T_0(\Gamma) > 0$, and $H$ with $T^{\eta} \le H \le T$,
\begin{equation}\label{eq:almostfinalbound}
N_X(\sigma, T, H) \ll_{\epsilon,\eta,\Gamma}  K^2 C^{\sigma} H^{1-2\left(\sigma-\frac{\delta}{2}\right)+ O(\frac{1}{K}) + \epsilon} T^{2\delta - 2\sigma + O(\frac{1}{K})} +  K^2 C^K T^{1-2K}.
\end{equation}
Choose $K= \log(T)$ and observe that $H^{O(\frac{1}{K})},T^{O(\frac{1}{K})}=O(1)$ and $K^2 = O_\epsilon (T^\epsilon)=O_{\epsilon,\eta} (H^\epsilon) $ for any $\epsilon>0$, and that the second term on the right of \eqref{eq:almostfinalbound} gets absorbed by the first one. This yields
$$
N_X(\sigma, T, H) \ll_{\epsilon,\eta,\Gamma} H^{1-2(\sigma-\frac{\delta}{2})+\epsilon} T^{2\delta - 2\sigma}. 
$$
In particular, we obtain
$$
N_X(\sigma,T) = N_X\left(\sigma,\frac{T}{2}, \frac{T}{2}\right) \ll_{\epsilon,\Gamma} T^{1+\delta - 4\left(\sigma-\frac{\delta}{2}\right)+\epsilon}.
$$
For these last two inequalities, we may assume $T\ge 1$ (as we did in the statement of Theorem \ref{main_thm}) simply by increasing the constant implied in the $\ll$ notation if necessary. This concludes the proof of Theorem \ref{main_thm}, conditional on Proposition \ref{prop:main_estimate}. The proof of the latter occupies the rest of this paper.

\begin{proof}[Proof of Proposition \ref{prop:main_estimate}]
Let $h, \tau_0, \tau_1 > 0$ be sufficiently small such that the operator
\begin{equation}\label{opfinal}
\mathcal{L}_{\tau_0,\tau_1,s} := \mathcal{L}_{\tau_1,s}\circ\mathcal{L}_{\tau_0,s} \colon H^2(\Omega(h)) \to H^2(\Omega(h))
\end{equation}
is well-defined. The parameters $\tau_0$ and $\tau_1$ will be chosen in terms of $T$ and $H$ during the proof. By Proposition \ref{prop:HSnorm}, we have
$$
\Vert \mathcal{L}_{\tau_0, \tau_1, s} \Vert_{\mathrm{HS},h}^2 = \sum_{b \in \mathcal{A}} 
\sum_{(\Aa_0,\Aa_1),(\Bb_0,\Bb_1)\in S_b}
\int_{\Omega_b(h)} g_{\Aa_0\Aa_1,\Bb_0\Bb_1}(z;s) \,\mathrm{dvol}(z),
$$
where
\[
S_b:=\{(\Aa_0,\Aa_1) : \Aa_1\in Y_b(\tau_1),\ \Aa_0\in Y_{\mathrm{first}(\Aa_1)}(\tau_0)\},
\]
and
$$
g_{\Aa_0\Aa_1,\Bb_0\Bb_1}(z;s) := \gamma'_{\Aa_0\Aa_1}(z)^s \overline{\gamma'_{\Bb_0\Bb_1}(z)^s} B_{\Omega(h)}(\gamma_{\Aa_0\Aa_1}(z), \gamma_{\Bb_0\Bb_1}(z)).
$$
Note that for all $t\in \mathbb{R}$ we have
$$
g_{\Aa, \Bb}(z;s+it) = g_{\Aa, \Bb}(z;s) e^{-it  \Phi_{\Aa,\Bb}(z)},
$$
where $ \Phi_{\Aa,\Bb}$ is the phase defined in \eqref{phase}, whence
$$
\Vert \mathcal{L}_{\tau_0, \tau_1, s+it} \Vert_{\mathrm{HS},h}^2 = \sum_{b \in \mathcal{A}} 
\sum_{(\Aa_0,\Aa_1),(\Bb_0,\Bb_1)\in S_b}
\int_{\Omega_b(h)} g_{\Aa_0\Aa_1,\Bb_0\Bb_1}(z;s) e^{-it  \Phi_{\Aa_0\Aa_1,\Bb_0\Bb_1}(z)} \,\mathrm{dvol}(z).
$$
Let $\varphi \in C^\infty(\mathbb{R})$ be a non-negative bump function supported on $[-2,2]$ with $\varphi \equiv 1$ on $[-1,1]$, and define
\[
\varphi_{T,H}(t) := \frac{1}{H} \, \varphi\left( \frac{t - T}{H} \right).
\]
Then $\varphi_{T,H}$ is supported on $[T - 2H, T + 2H]$ and equals $1/H$ on $[T - H, T + H]$, so that
\[
\frac{1}{H} \int_{T - H}^{T + H} \left\| \mathcal{L}_{\tau_0, \tau_1, s + it} \right\|_{\mathrm{HS}, h}^2 \, dt
\le \int_{-\infty}^{\infty} \varphi_{T,H}(t) \left\| \mathcal{L}_{\tau_0, \tau_1, s + it} \right\|_{\mathrm{HS}, h}^2 \, dt
=:\Sigma.
\]
Expanding the HS-norm using the previous formula, we get
\begin{equation}\label{eq:sigma-expansion}
\Sigma = \sum_{b \in \mathcal{A}} 
\sum_{(\Aa_0,\Aa_1),(\Bb_0,\Bb_1)\in S_b}
\int_{\Omega_b(h)} g_{\Aa_0\Aa_1,\Bb_0\Bb_1}(z;s) \, \widehat{\varphi_{T,H}}(\Phi_{\Aa_0\Aa_1,\Bb_0\Bb_1}(z)) \,\mathrm{dvol}(z).
\end{equation}

From now on, let
$$
\tau_0 = \widetilde{C} h,
$$
where $\widetilde{C} > 0$ is the constant from Lemma \ref{lem:separation}. If $\Aa_0\nsim_{\mathrm{pre}}\Bb_0$, apply Lemma~\ref{lem:separation} with $z_1=\gamma_{\Aa_1}(z)$ and $z_2=\gamma_{\Bb_1}(z)$. It follows that $\gamma_{\Aa_0\Aa_1}(z)$ and $\gamma_{\Bb_0\Bb_1}(z)$ lie in different connected components of $\Omega(h)$, so the corresponding Bergman kernel vanishes by Lemma~\ref{lem:BergmanBasicProp}~(i). Hence, the sum \eqref{eq:sigma-expansion} reduces to terms where $\Aa_0\sim_{\mathrm{pre}}\Bb_0$. Recall that $\Aa\sim_{\mathrm{suf}}\Bb$ means that one of $\Aa$ and $\Bb$ is a suffix of the other. Let $F$ be the finite set from Lemma~\ref{lem:boundedTransitions}. Then the sum is over pairs that are contained in one of the following sets:
\begin{align*}
Q_b^{(1)} &:= \{ (\Aa_0\Aa_1,\Bb_0\Bb_1) : (\Aa_0,\Aa_1),(\Bb_0,\Bb_1)\in S_b, \Aa_0\sim_{\mathrm{pre}}\Bb_0,\ \Aa_1\sim_{\mathrm{suf}}\Bb_1 \}, \\
Q_b^{(2)} &:= \{ (\Cc\Aa_1,\Cc\Bb_1) : (\Cc,\Aa_1),(\Cc,\Bb_1)\in S_b,\ \Aa_1\nsim_{\mathrm{suf}}\Bb_1 \}, \\
Q_b^{(3)} &:= \{ (\Cc\Aa_1,\Cc\textbf{v}\Bb_1) : \textbf{v}\in F,\ (\Cc,\Aa_1),(\Cc\textbf{v},\Bb_1)\in S_b,\ \Aa_1\nsim_{\mathrm{suf}}\Bb_1 \} \\
&\quad\cup\{ (\Cc\textbf{v}\Aa_1,\Cc\Bb_1) : \textbf{v}\in F,\ (\Cc\textbf{v},\Aa_1),(\Cc,\Bb_1)\in S_b,\ \Aa_1\nsim_{\mathrm{suf}}\Bb_1 \}.
\end{align*}
Put $Q_b:=Q_b^{(1)}\cup Q_b^{(2)}\cup Q_b^{(3)}$.
Since $\Sigma\ge0$, the triangle inequality gives
$$
\Sigma \le \Sigma^{(1)} + \Sigma^{(2)}+\Sigma^{(3)},
$$
where
$$
\Sigma^{(\ell)} := \sum_{b \in \mathcal{A}} 
\sum_{(\Aa,\Bb) \in Q_b^{(\ell)}}
\left|\int_{\Omega_b(h)} g_{\Aa,\Bb}(z;s) \widehat{\varphi_{T,H}}(\Phi_{\Aa,\Bb}(z)) \,\mathrm{dvol}(z)\right|, \quad \ell\in \{1,2,3\}.
$$

Throughout, let $T\gg 1$ and assume $h=T^{-1}$, $\eta > 0$, $T^{\eta}\le H \le  T$, $\mathrm{Re}(s)\ge \sigma>0$, and $\vert \mathrm{Im}(s) \vert\le  T$, as in the statement of the Proposition. Moreover, $C>0$ is a constant depending only on $\Gamma$ whose precise value changes from place to place. All implied constants are allowed to depend on $\Gamma.$

Lemma \ref{lem:welldefinedness} ensures that if $\tau_0 > 0$ and $\tau_1 > 0$ are sufficiently small, then for all $(\Aa,\Bb)\in Q_b$ and $z\in \Omega_b(h)$ we have
$$
\mathrm{dist}( \gamma_{\Aa}(z), \partial \Omega(h) ) \ge h/2, \quad \mathrm{dist}( \gamma_{\Bb}(z), \partial \Omega(h) )\ge h/2.
$$
Combining this with Lemma \ref{lem:UpperBoundBergman}, we get
\begin{equation}\label{est:BKpointwise}
\left\vert B_{\Omega(h)}(\gamma_{\Aa}(z), \gamma_{\Bb}(z)) \right\vert \le  \frac{1}{\pi\, \mathrm{dist}( \gamma_{\Aa}(z), \partial \Omega(h) )\, \mathrm{dist}( \gamma_{\Bb}(z), \partial \Omega(h) ) } \ll h^{-2}.
\end{equation}
Lemmas \ref{lem:boundsderivatives} and \ref{lem:YZ} imply that for each $b\in \mathcal{A}$ and $(\Aa,\Bb)\in Q_b$, we have 
$$
\Upsilon_{\Aa}  \asymp \tau_0 \tau_1, \quad \Upsilon_{\Bb}  \asymp \tau_0 \tau_1.
$$
Combining this with Lemma~\ref{lem:crucial_estimate}, we have for all $\mathrm{Re}(s) > 0,$
\begin{equation}
\vert \gamma'_{\Aa}(z)^s \vert \le ( C\tau_0 \tau_1 )^{\mathrm{Re}(s)} e^{C h |\mathrm{Im}(s)|}.
\end{equation}
Since $h=T^{-1}$ and $|\mathrm{Im}(s)|\le T$, the exponential factor is uniformly bounded. Assuming $\tau_0, \tau_1$ are small enough so that $C\tau_0 \tau_1 < 1$, we obtain for all $\mathrm{Re}(s) \ge \sigma > 0,$
\begin{equation}\label{final:bound_gamma_a}
\vert \gamma'_{\Aa}(z)^s \vert \ll ( C\tau_0 \tau_1 )^{\sigma}.
\end{equation}
The same estimate holds for $\Bb$. Combining \eqref{final:bound_gamma_a} and \eqref{est:BKpointwise}, we deduce that for all $(\Aa,\Bb)\in Q_b$ and $z\in \Omega_b(h)$:
\begin{equation}\label{est:pointwisegab}
\vert g_{\Aa, \Bb}(z;s) \vert \ll  ( C\tau_0 \tau_1 )^{2\sigma} h^{-2}.
\end{equation}
Meanwhile, Lemma \ref{lem:YZ} gives
\begin{equation}\label{final:bound_ytau}
\vert Y(\tau)\vert \asymp \tau^{-\delta}.
\end{equation}
Recall from our choices of $\tau_0$ and $h$ that
\begin{equation}\label{final:bound_tau0}
\tau_0 \asymp h = T^{-1}.
\end{equation}
By Lemma~\ref{lem:boundedTransitions}, for each $(\Aa_0,\Aa_1)\in S_b$, there are $O_\Gamma(1)$ pairs in $Q_b^{(1)}$ with first component $\Aa_0\Aa_1$. Hence, $|Q_b^{(1)}| \ll_\Gamma |S_b| \ll |Y(\tau_0)| |Y(\tau_1)|$. Together with the first estimate of Proposition \ref{prop:FOURIER} and \eqref{est:pointwisegab}, we get
\begin{align}
\Sigma^{(1)}
&\ll |Y(\tau_0)| |Y(\tau_1)| (C \tau_0\tau_1)^{2\sigma} h^{-\delta} \\
&\ll ( C\tau_0 \tau_1 )^{-\delta + 2\sigma} h^{-\delta}&&\text{by \eqref{final:bound_ytau}} \\
&\ll C^\sigma T^{2\delta-2 \sigma} \tau_1 ^{-\delta + 2\sigma} &&\text{by \eqref{final:bound_tau0}} \label{final1}.
\end{align}

Now we estimate $\Sigma^{(2)}$. Recall that $Q_b^{(2)}$ consists of pairs $(\Aa,\Bb)\in \mathcal{W}\times \mathcal{W}$ with $\Aa=\Cc\Aa_1$ and $\Bb=\Cc\Bb_1$ such that
\begin{itemize}
\item $\Cc\in Y(\tau_0)$ and $\Aa_1, \Bb_1 \in Y(\tau_1)$,
\item $\Cc \to \Aa_1, \Bb_1 \to b$, and
\item $\Aa_1\nsim_{\mathrm{suf}}\Bb_1$.
\end{itemize}
Let $\Dd$ be the maximal common suffix of $\Aa_1$ and $\Bb_1$, and write $\Aa_1=\Ee\Dd$ and $\Bb_1=\Ff\Dd$. The condition $\Aa_1\nsim_{\mathrm{suf}}\Bb_1$ ensures that $\Ee$ and $\Ff$ are non-empty words. The maximality of $\Dd$ gives $\Ee\to\overline{\Ff}$, while $\Cc\to\Ee$ and $\overline{\Ff}\to\overline{\Cc}$. Therefore,
\[
\mathsf{red}(\Aa\overline{\Bb})
=\Cc\Ee\overline{\Ff}\overline{\Cc}
=\Cc\,\mathsf{red}(\Aa_1\overline{\Bb}_1)\,\overline{\Cc}.
\]
Applying Parts \eqref{part:mirrorEst} and \eqref{part:almostMultip} of Lemma \ref{lem:boundsderivatives}, we find
\begin{itemize}
\item $\Upsilon_{\mathsf{red}(\Aa\overline{\Bb})} \asymp \Upsilon_\Cc^2 \Upsilon_{\mathsf{red}(\Aa_1 \overline{\Bb}_1)} , $
\item $\Upsilon_{\Aa} \asymp \Upsilon_\Cc  \Upsilon_{\Aa_1},$ and
\item $\Upsilon_{\Bb} \asymp \Upsilon_\Cc  \Upsilon_{\Bb_1}.$
\end{itemize}
This implies that for all $(\Aa,\Bb)\in Q_b^{(2)}$
\begin{equation}\label{final:bound_Dab}
\mathcal{D}_{\Aa,\Bb} = \left(  \frac{\Upsilon_{\Aa} \Upsilon_{\Bb} }{\Upsilon_{\mathsf{red}(\Aa\overline{\Bb})}}\right)^{1/2} \asymp \left(  \frac{  \Upsilon_{\Aa_1} \Upsilon_{\Bb_1} }{  \Upsilon_{\mathsf{red}(\Aa_1 \overline{\Bb}_1)}  }\right)^{1/2} \asymp \tau_1 \Upsilon_{\mathsf{red}(\Aa_1 \overline{\Bb}_1)}^{-1/2}.
\end{equation}
Using Proposition \ref{prop:FOURIER}, we can now estimate the off-diagonal contribution as follows:
\begin{align*}
\Sigma^{(2)} &\le \sum_{b\in \mathcal{A}}  
\sum_{ (\Aa,\Bb) \in Q_b^{(2)} } 
\left|\int_{\Omega_b(h)} g_{\Aa,\Bb}(z;s) \widehat{\varphi_{T,H}}(\Phi_{\Aa,\Bb}(z)) \,\mathrm{dvol}(z)\right|\\
&\ll_{\epsilon, \eta} \sum_{b\in \mathcal{A}}  \sum_{ (\Aa,\Bb) \in Q_b^{(2)} }  h^{-\delta + 2}  \mathcal{D}_{\Aa,\Bb}^{-\delta} H^{-\delta+\epsilon} \left\Vert g_{\Aa, \Bb}(\cdot;s) \right\Vert_{\infty,\Omega_b(h)} &&\text{by Prop. \ref{prop:FOURIER}}\\
&\ll \sum_{b\in \mathcal{A}} \sum_{ (\Aa,\Bb) \in Q_b^{(2)} } (C \tau_0 \tau_1)^{2\sigma} h^{-\delta}  \mathcal{D}_{\Aa,\Bb}^{-\delta} H^{-\delta+\epsilon} &&\text{by \eqref{est:pointwisegab}}\\
&\ll \sum_{b\in \mathcal{A}} \sum_{ (\Aa,\Bb) \in Q_b^{(2)} } (C \tau_0 \tau_1)^{2\sigma} \tau_1^{-\delta} h^{-\delta} H^{-\delta+\epsilon} \Upsilon_{\mathsf{red}(\Aa_1 \overline{\Bb}_1)}^{\delta/2} &&\text{by \eqref{final:bound_Dab}}\\
&\ll \vert Y(\tau_0)\vert (C \tau_0 \tau_1)^{2\sigma} \tau_1^{-\delta} h^{-\delta} H^{-\delta+\epsilon} \left(  \sum_{\Aa_1,\Bb_1\in Y(\tau_1)}   \Upsilon_{\mathsf{red}(\Aa_1 \overline{\Bb}_1)}^{\delta/2} \right)\\
&\ll (C \tau_0 \tau_1)^{-\delta + 2\sigma} h^{-\delta} H^{-\delta+\epsilon} \left(  \sum_{\Aa_1,\Bb_1\in Y(\tau_1)}   \Upsilon_{\mathsf{red}(\Aa_1 \overline{\Bb}_1)}^{\delta/2} \right) &&\text{by \eqref{final:bound_ytau}}\\
&\ll C^{\sigma}  T^{2\delta-2\sigma} \tau_1^{-\delta +2\sigma}   H^{-\delta+\epsilon} \left(  \sum_{\Aa_1,\Bb_1\in Y(\tau_1)}   \Upsilon_{\mathsf{red}(\Aa_1 \overline{\Bb}_1)}^{\delta/2} \right) &&\text{by \eqref{final:bound_tau0}}.
\end{align*}
Since $\Upsilon_{\Aa_1}\asymp\Upsilon_{\Bb_1}\asymp\tau_1$, Lemma~\ref{lem:weird_sum} with $\mathcal{U}=\mathcal{V}=Y(\tau_1)$ and $\alpha=\delta/2$ gives
\begin{equation}\label{final2}
\Sigma^{(2)} \ll_{\epsilon, \eta} C^{\sigma}  T^{2\delta-2\sigma}  H^{-\delta+\epsilon}  \tau_1^{-2\delta +2\sigma-\epsilon}. 
\end{equation}

Let us now turn to $\Sigma^{(3)}$. Consider a pair $(\Aa,\Bb)=(\Cc\Aa_1,\Cc\textbf{v}\Bb_1)$ in the first part of $Q_b^{(3)}$, and put $\widetilde{\Bb}_1:=\textbf{v}\Bb_1$. Then $\Aa_1\nsim_{\mathrm{suf}}\widetilde{\Bb}_1$. Indeed, if $\widetilde{\Bb}_1$ were a suffix of $\Aa_1$, then so would $\Bb_1$. Conversely, if $\Aa_1$ were a suffix of $\widetilde{\Bb}_1$, then, since $\widetilde{\Bb}_1$ ends with $\Bb_1$, the words $\Aa_1$ and $\Bb_1$ would be suffix-comparable. Both cases contradict $\Aa_1\nsim_{\mathrm{suf}}\Bb_1$. Thus, the preceding reduction argument applies to $\Aa_1$ and $\widetilde{\Bb}_1$. Since $\Cc\to\Aa_1$, $\Cc\to\widetilde{\Bb}_1$, and $F$ is finite,
\[
\Upsilon_{\Aa_1}\asymp_F\Upsilon_{\widetilde{\Bb}_1}\asymp_F\tau_1.
\]
\[
\Upsilon_{\mathsf{red}(\Aa\overline{\Bb})}
\asymp_F
\Upsilon_\Cc^2\Upsilon_{\mathsf{red}(\Aa_1\overline{\Bb}_1\overline{\textbf{v}})}.
\]
Using Lemma~\ref{lem:boundsderivatives} and Lemma~\ref{lem:YZ}, we obtain
\[
\mathcal{D}_{\Aa,\Bb}
\asymp_F
\tau_1\Upsilon_{\mathsf{red}(\Aa_1\overline{\Bb}_1\overline{\textbf{v}})}^{-1/2}.
\]
The second part of $Q_b^{(3)}$ is estimated in the same way after interchanging $\Aa$ and $\Bb$. For each $\textbf{v}\in F$, put
\[
\mathcal{V}_{\textbf{v}}:=\{\textbf{v}\Bb_1:\Bb_1\in Y(\tau_1),\ \textbf{v}\to\Bb_1\}.
\]
Since $F$ is finite, Lemmas~\ref{lem:boundsderivatives} and \ref{lem:YZ} give $\Upsilon_{\Bb}\asymp_F\tau_1$ for all $\Bb\in\mathcal{V}_{\textbf{v}}$. Thus Lemma~\ref{lem:weird_sum}, applied with $\mathcal{U}=Y(\tau_1)$, $\mathcal{V}=\mathcal{V}_{\textbf{v}}$, and $\alpha=\delta/2$, gives
\[
\sum_{\Aa_1\in Y(\tau_1),\,\Bb\in\mathcal{V}_{\textbf{v}}}
\Upsilon_{\mathsf{red}(\Aa_1\overline{\Bb})}^{\delta/2}
\ll_{\epsilon,F,\Gamma}\tau_1^{-\delta-\epsilon}.
\]
Summing over $\textbf{v}\in F$ and using the same chain of inequalities as above, we get
\begin{align*}
\Sigma^{(3)}
&\ll_{\epsilon,\eta,F} |Y(\tau_0)|(C\tau_0\tau_1)^{2\sigma}\tau_1^{-\delta}h^{-\delta}H^{-\delta+\epsilon}
\sum_{\textbf{v}\in F}\ \sum_{\substack{\Aa_1,\Bb_1\in Y(\tau_1)\\ \textbf{v}\to\Bb_1}}
\Upsilon_{\mathsf{red}(\Aa_1\overline{\Bb}_1\overline{\textbf{v}})}^{\delta/2} \\
&\ll_{\epsilon,\eta} C^{\sigma}T^{2\delta-2\sigma}H^{-\delta+\epsilon}\tau_1^{-2\delta+2\sigma-\epsilon}.
\end{align*}
Combining the bounds \eqref{final1}, \eqref{final2}, and the estimate for $\Sigma^{(3)}$, we obtain
\begin{align*}
\Sigma &\ll_\epsilon C^\sigma T^{2\delta-2 \sigma} \tau_1 ^{-\delta + 2\sigma} +    C^{\sigma}  T^{2\delta-2\sigma}  H^{-\delta+\epsilon}  \tau_1^{-2\delta +2\sigma-\epsilon}   \\
&= C^\sigma T^{2\delta-2 \sigma} \tau_1^{-\delta + 2\sigma} \left( 1 + H^{-\delta+\epsilon}  \tau_1^{-\delta-\epsilon} \right).
\end{align*}
It remains to choose $\tau_1$ optimally. This may be done by taking 
$$
\tau_1 = c H^{-1}
$$
for some constant $c>0$ sufficiently small to ensure that the parameter $\tau_1$ is admissible in all the above estimates. Inserting this choice into the previous estimate yields
$$
\frac{1}{2H} \int_{T-H}^{T+H} \Vert \mathcal{L}_{\tau_0, \tau_1, s+it} \Vert_{\mathrm{HS},h}^2 dt \ll_{\epsilon, \eta} C^\sigma T^{2\delta-2 \sigma} H^{-2\left(\sigma-\frac{\delta}{2}\right)+\epsilon}.
$$
This completes the proof.
\end{proof}

\normalem
\bibliography{improved}
\bibliographystyle{amsplain}

\end{document}